\documentclass[a4paper,reqno,10.5pt, oneside]{amsart}

\usepackage{amssymb}
\usepackage{amstext}
\usepackage{amsmath}
\usepackage{amscd}
\usepackage{amsthm}
\usepackage{amsfonts}
\usepackage{enumerate}
\usepackage{graphicx}
\usepackage{latexsym}
\usepackage{mathrsfs}
\usepackage{mathtools}
\usepackage{here}

\usepackage{tikz}
\usepackage{tikz-cd}
\usetikzlibrary{arrows}
\usetikzlibrary{decorations.markings}
\usetikzlibrary{patterns,decorations.pathreplacing}

\makeatletter
\pgfdeclarepatternformonly[\LineSpace]{my north west lines}{\pgfqpoint{-1pt}{-1pt}}{\pgfqpoint{\LineSpace}{\LineSpace}}{\pgfqpoint{\LineSpace}{\LineSpace}}
{
    \pgfsetcolor{\tikz@pattern@color} 
    \pgfsetlinewidth{0.6pt}
    \pgfpathmoveto{\pgfqpoint{0pt}{\LineSpace}}
    \pgfpathlineto{\pgfqpoint{\LineSpace + 0.1pt}{- 0.1pt}}
    \pgfusepath{stroke}
}
\makeatother 
\newdimen\LineSpace
\tikzset{
    line space/.code={\LineSpace=#1},
    line space=8.5pt
}

\usepackage{hyperref}

\newtheorem{theorem}{Theorem}[section]

\newtheorem{corollary}[theorem]{Corollary}
\newtheorem{lemma}[theorem]{Lemma}
\newtheorem{proposition}[theorem]{Proposition}
\newtheorem{definition-proposition}[theorem]{Definition-Proposition}

\theoremstyle{definition}
\newtheorem{definition}[theorem]{Definition}
\newtheorem{example}[theorem]{Example}
\newtheorem{observation}[theorem]{Observation}

\newtheorem{settings}[theorem]{Settings}
\newtheorem{notation}[theorem]{Notation}

\newtheorem{remark}[theorem]{Remark}

\makeatletter
  
  \@addtoreset{equation}{section}
\makeatother

\newcommand{\rmb}{\mathrm{b}}

\newcommand{\rmD}{\mathrm{D}}

\newcommand{\rmV}{\mathrm{V}}

\newcommand{\rmX}{\mathrm{X}}
\newcommand{\rmY}{\mathrm{Y}}

\newcommand{\calA}{\mathcal{A}}

\newcommand{\calC}{\mathcal{C}}
\newcommand{\calD}{\mathcal{D}}

\newcommand{\calG}{\mathcal{G}}
\newcommand{\calH}{\mathcal{H}}

\newcommand{\calL}{\mathcal{L}}
\newcommand{\calM}{\mathcal{M}}
\newcommand{\calN}{\mathcal{N}}

\newcommand{\calP}{\mathcal{P}}

\newcommand{\calT}{\mathcal{T}}
\newcommand{\calU}{\mathcal{U}}

\newcommand{\calW}{\mathcal{W}}

\newcommand{\fkp}{\mathfrak{p}}

\newcommand{\ZZ}{\mathbb{Z}}
\newcommand{\QQ}{\mathbb{Q}}
\newcommand{\RR}{\mathbb{R}}

\newcommand{\TT}{\mathbb{T}}

\newcommand{\sfM}{\mathsf{M}}
\newcommand{\sfN}{\mathsf{N}}

\newcommand{\bfa}{\mathbf{a}}
\newcommand{\bfb}{\mathbf{b}}
\newcommand{\bfv}{\mathbf{v}}

\newcommand{\kk}{\Bbbk}

\newcommand{\Spec}{\operatorname{Spec}}

\newcommand{\Hom}{\operatorname{Hom}}

\newcommand{\Ker}{\operatorname{Ker}}

\newcommand{\GL}{\operatorname{GL}}

\newcommand{\Sym}{\operatorname{Sym}}

\newcommand{\Cl}{\operatorname{Cl}}

\newcommand{\End}{\operatorname{End}}
\newcommand{\gldim}{\mathrm{gl.dim}}
\newcommand{\pdim}{\mathrm{proj.dim}}

\newcommand{\add}{\mathsf{add}}

\newcommand{\mc}{\mathsf{mod}}
\newcommand{\refl}{\mathsf{ref}}
\newcommand{\coh}{\mathsf{coh}}

\begin{document}

\title[NCCRs of Hibi rings with small class group]{Non-commutative crepant resolutions of Hibi rings \\ with small class group}  

\author[Y. Nakajima]{Yusuke Nakajima} 

\address[Y. Nakajima]{Kavli Institute for the Physics and Mathematics of the Universe (WPI), UTIAS, The University of Tokyo, Kashiwa, Chiba 277-8583, Japan} 
\email{yusuke.nakajima@ipmu.jp}


\subjclass[2010]{Primary 13C14, 16S38; Secondary 06A11, 14M25, 14E15.}
\keywords{Non-commutative crepant resolutions, Hibi rings, class groups} 

\maketitle

\begin{abstract} 
In this paper, we study splitting (or toric) non-commutative crepant resolutions (= NCCRs) of some toric rings. 
In particular, we consider Hibi rings, which are toric rings arising from partially ordered sets, and show that Gorenstein Hibi rings with class group $\ZZ^2$ have a splitting NCCR. 

In the appendix, we also discuss Gorenstein toric rings with class group $\ZZ$, in which case the existence of splitting NCCRs is already known. 
We especially observe the mutations of modules giving splitting NCCRs for the three dimensional case, and show the connectedness of the exchange graph. 
\end{abstract}



\section{\bf Introduction} 
\label{sec_intro}

In this paper, we discuss non-commutative crepant resolutions for some toric rings. Thus, we first recall the notion of non-commutative crepant resolutions  defined by Van den Bergh \cite{VdB2}. 

\begin{definition}
\label{def_NCCR}
Let $R$ be a Cohen-Macaulay (= CM) normal domain, and $M$ be a non-zero reflexive $R$-module. Let $\Lambda\coloneqq\End_R(M)$. 
We say that $\Lambda$ is a \emph{non-commutative crepant resolution} (= \emph{NCCR}) of $R$ 
(or $M$ \emph{gives an NCCR} of $R$) if $\gldim \Lambda_\mathfrak{p}=\dim R_\mathfrak{p}$ for all $\mathfrak{p}\in\Spec R$ and $\Lambda$ is a maximal Cohen-Macaulay (= MCM) $R$-module.
\end{definition}

\begin{remark}
\label{rem_NCCR}
We note some remarks concerning the definition of NCCRs. 
\begin{enumerate}[\rm (a)]
\item When $R$ is a Gorenstein normal domain, we can relax the above definition. 
In such a situation, $\Lambda$ is an NCCR of $R$  if and only if $\gldim \Lambda<\infty$ and $\Lambda$ is an MCM $R$-module (see e.g., \cite[Lemma~4.2]{VdB2}, \cite[Lemma~2.23]{IW1}). 
\item It was shown in \cite{DITW} that if a CM normal domain $R$ has an NCCR, then $R$ is $\QQ$-Gorenstein. 
Thus, if the class group of $R$ is a free abelian group and $R$ is not Gorenstein, then $R$ does not have an NCCR. 
In this paper, we will mainly consider an NCCR of a CM normal domain $R$ whose class group is free abelian, thus we often assume that $R$ is Gorenstein. 
\item In addition, we say that an NCCR $\End_R(M)$ is \emph{splitting} if $M$ is a finite direct sum of rank one reflexive $R$-modules. 
A splitting NCCR is also called ``\emph{toric NCCR}" when $R$ is a toric ring (see \cite{Boc2}). 
\end{enumerate}
\end{remark}

Some NCCRs can be obtained as the endomorphism ring of a tilting bundle on a crepant resolution $Y$ of $\Spec R$, in which case a tilting bundle $\calT$ induces a derived equivalence $\rmD^\rmb(\coh Y)\cong\rmD^\rmb(\mc\End(\calT))$. 
By this equivalence we can consider an NCCR as a non-commutaive analogue of a crepant resolution, and it gives an interaction between birational geometry and representation theory of algebras. 
From the viewpoint of representation theory, NCCRs are also related with cluster tilting theory, higher dimensional Auslander-Reiten theory etc (see \cite{Iya,IR,IW1} for example). 
One of the important problems concerning NCCRs is the existence of an NCCR for a given singularity. 
For some nice cases, NCCRs have been constructed in several literatures, see e.g., \cite{Bro,BLVdB,Har,HN,IU,IW1,IW2,SpVdB,VdB2} and the survey article \cite{Leu}. 
In this paper, we especially focus on NCCRs of toric rings. The existence of NCCRs is known for several toric rings. 
For example, the case of quotient singularities associated with a finite abelian subgroup in $\GL(\kk,d)$ \cite{VdB2,IW1}, Gorenstein toric rings whose class group is $\ZZ$ \cite{VdB2} (see also Appendix~\ref{sec_dimer}), 
$3$-dimensional Gorenstein toric rings \cite{Bro,IU,SpVdB_toric2}, and some higher dimensional Gorenstein toric rings \cite{SpVdB,HN}. 
In general, the existence of NCCRs for toric rings is still open. 

On the other hand, NCCRs of toric rings mentioned above are splitting ones. 
Thus, we hoped that any toric ring has a splitting NCCR, but recently \v{S}penko and Van den Bergh gave the following example. 

\begin{proposition}[see {\cite[Example~3.3]{SpVdB_toric1}}]
Let $G\coloneqq(\kk^\times)^2$ be the two dimensional algebraic torus and $S\coloneqq\kk[x_1,\cdots,x_6]$ be a polynomial ring.  
We consider the action of $G$ on $S$ defined by $g\cdot x_i\coloneqq \beta_i(g)x_i$ for $g\in G$ and $i=1,\cdots,6$, 
in which case $R\coloneqq S^G$ is a toric ring with $\dim R=4$ and the class group $\Cl(R)$ is isomorphic to $\ZZ^2$. 
Here, $\beta_i$ is the character of $G$ corresponding to the weight $(1, 1), (-1, 0), (0,-1), (-3,-3), (3, 0), (0, 3)$ respectively (see Subsection~\ref{sec_preliminary_toric} for more details). 
Then, such a toric ring $R$ has an NCCR but does not have any splitting NCCR. 
\end{proposition}

Thus, it is interesting to ask when a Gorenstein toric ring $R$ with $\Cl(R)\cong\ZZ^2$ has a splitting NCCR. 
In this paper, we consider Hibi rings which are toric rings arising from partially ordered sets, and show the existence of a splitting NCCR for the following Hibi rings. 

\begin{theorem}[{see Theorem~\ref{main_thm_HibiZ2}}]
\label{main_intro}
Let $R$ be a Gorenstein Hibi ring with the class group $\Cl(R)\cong\ZZ^2$. Then, $R$ has a splitting NCCR. 
\end{theorem}

An idea for constructing a splitting NCCR of a toric ring is to consider a nice class of rank one MCM modules, called \emph{conic modules} (see Section~\ref{sec_preliminary}). 
Conic modules were well studied in \cite{Bru,BG1,SmVdB}, and it is known that the number of those is finite up to isomorphism. 
Furthermore, the endomorphism ring of the direct sum of all conic modules has finite global dimension (see \cite[Proposition~1.8]{SpVdB}, \cite[Theorem~6.1]{FMS}), but in general this endomorphism ring is not an MCM module, and hence not an NCCR. 
In particular, taking all conic modules is too big to allow the endomorphism ring to be an MCM module. 
A naive idea for constructing an NCCR is to consider a part of conic modules. 
In fact, an NCCR of a ``quasi-symmetric" toric ring given in \cite{SpVdB}, and an NCCR of the Segre products of polynomial rings given in \cite{HN} are constructed as the endomorphism ring of a part of conic modules. 
In this paper, we will use the same strategy for constructing an NCCR of Gorenstein Hibi rings with the class group $\ZZ^2$, and the resulting NCCR is splitting by this construction. 
In particular, we give the precise description of a module giving a splitting NCCR of these Hibi rings. 
Thus, starting from this module giving an NCCR, we obtain other modules giving NCCRs using the operation which we call the \emph{mutation}. 
More precisely, we choose a direct summand of a module giving an NCCR, and replace the chosen direct summand by another module using the approximation theory of MCM modules. Then, the resulting module also gives an NCCR (see \cite[Section~6]{IW1}, \cite[Section~4]{HN}). 
The mutated module is not necessarily isomorphic to the original one, and we sometimes have infinitely many modules giving NCCRs by repeating the mutations. 

We note that to show a given endomorphism ring $\Lambda$ is an NCCR, we have to show that the global dimension of $\Lambda$ is finite and $\Lambda$ is an MCM module. 
In this paper, Lemma~\ref{key_lem1} and \ref{key_lem2} are main ingredients for showing the finiteness of global dimension. 
Since these hold for any toric ring, the method used in this paper might be valid for constructing NCCRs of other toric rings, 
for example Hibi rings with the class group $\ZZ^m$ where $m>2$. 
However, the combinatorics for checking the assumption of Lemma~\ref{key_lem2} would be complicated in general. 
In addition, when we check whether a given module is an MCM module or not, we often consider the vanishing of the local cohomology, 
and it can be understood by using the combinatorics (see e.g., \cite{VdB1}). However, that combinatorics is also complicated in general. 

\medskip

The content of this paper is the following. In Section~\ref{sec_preliminary}, we prepare some notations regarding toric rings and their divisorial ideals. We then restrict  to Hibi rings and review some known facts.
In Section~\ref{sec_HibiZ2}, we first classify Gorenstein Hibi rings with class group $\ZZ^2$, and then we construct splitting NCCRs for those Hibi rings.
When we construct such a splitting NCCR, we have to check whether a given rank one reflexive module is MCM or not. 
Thus, we will give the complete list of rank one MCM modules for Gorenstein Hibi rings with class group $\ZZ^2$ in Section~\ref{sec_MCM_HibiZ2}. 

As we mentioned, any Gorenstein toric ring with the class group $\ZZ$ admits a splitting NCCR. 
In Appendix~\ref{sec_dimer}, we further study this class and give the precise description of modules giving splitting NCCRs. 
We then restrict to $3$-dimensional Gorenstein toric rings, in which case splitting NCCRs are obtained from consistent dimer models. 
In particular, we show that the exchange graph of the mutations, which is the graph whose vertices are modules giving splitting NCCRs 
and we draw an edge between two vertices if the corresponding modules are transformed into each other by a single mutation, 
is connected for $3$-dimensional Gorenstein toric rings with the class group $\ZZ$. 
This gives a partial affirmative answer to \cite[Question~6.4]{Nak1}. 

\subsection*{Notation and Conventions} 
Throughout this paper, we suppose that $\kk$ is an algebraically closed field. 
For an $R$-module $M$, we denote by $\add_RM$ the category consisting of direct summands of finite direct sums of some copies of $M$. 
We say that $M\in\mc R$ is a \emph{generator} if $R\in\add_R M$. 
We say that $M=\bigoplus_iM_i$ is \emph{basic} if $M_i$'s are mutually non-isomorphic. 

\section{Preliminaries}
\label{sec_preliminary}

\subsection{Preliminaries on toric rings}
\label{sec_preliminary_toric}

In this paper, we mainly study a toric ring $R$ whose class group $\Cl(R)$ is $\ZZ^r$. 
Such an $R$ is described in several ways as we will see below. 
Let $\sfN\cong\ZZ^d$ be a lattice of rank $d$ and $\sfM\coloneqq\Hom_\ZZ(\sfN, \ZZ)$ be the dual lattice of $\sfN$. 
Let $\sfN_\RR\coloneqq\sfN\otimes_\ZZ\RR$ and $\sfM_\RR\coloneqq\sfM\otimes_\ZZ\RR$. 
We denote the natural inner product by $\langle\;,\;\rangle:\sfM_\RR\times\sfN_\RR\rightarrow\RR$. 
Let 
$$
\tau\coloneqq\mathrm{Cone}(\bfv_1, \cdots, \bfv_n)=\RR_{\ge 0}\bfv_1+\cdots +\RR_{\ge 0}\bfv_n
\subset\sfN_\RR 
$$
be a strongly convex rational polyhedral cone of dimension $d$ generated by $\bfv_1, \cdots, \bfv_n\in\ZZ^d$ where $n\ge d$. 
We assume that this system of generators is minimal. 
For each generator $\bfv_i$, we consider the linear form $\sigma_i(-)\coloneqq\langle-, \bfv_i\rangle$, and denote $\sigma(-)\coloneqq(\sigma_1(-),\cdots,\sigma_n(-))$. 
We then consider the dual cone $\tau^\vee$: 
$$
\tau^\vee\coloneqq\{{\bf x}\in\sfM_\RR \mid \sigma_i({\bf x})\ge0 \text{ for all } i=1,\cdots,n \}. 
$$
Using this cone, we define the toric ring 
$$
R\coloneqq \kk[\tau^\vee\cap\sfM]=\kk[t_1^{m_1}\cdots t_d^{m_d}\mid (m_1, \cdots, m_d)\in\tau^\vee\cap\sfM]. 
$$
It is known that $R$ is a $d$-dimensional Cohen-Macaulay (= CM) normal domain. 

Then, for each $\bfa=(a_1, \cdots, a_n)\in\RR^n$, we set 
$$
\TT(\bfa)\coloneqq\{{\bf x}\in\sfM \mid (\sigma_1({\bf x}), \cdots, \sigma_n({\bf x}))\ge(a_1, \cdots, a_n)\}. 
$$
Then, we define the divisorial ideal (rank one reflexive $R$-module) $T(\bfa)$ generated by all monomials whose exponent vector is in $\mathbb{T}(\bfa)$. 
By the definition, we have $T(\bfa)=T(\ulcorner \bfa\urcorner)$, where $\ulcorner \; \urcorner$ means the round up 
and $\ulcorner \bfa\urcorner=(\ulcorner a_1\urcorner, \cdots, \ulcorner a_n\urcorner)$. 
Any divisorial ideal of $R$ takes this form (see e.g., \cite[Theorem~4.54]{BG2}), 
and hence each divisorial ideal is represented by $\bfa\in\ZZ^n$.  
We denote the class group of $R$ by $\Cl(R)$, and in this paper we assume that $\Cl(R)\cong\ZZ^r$. 
It is known that there exists the exact sequence
\begin{equation}
\label{cl_seq}
0\rightarrow\ZZ^d\xrightarrow{\sigma(-)}\ZZ^n\rightarrow\Cl(R)\rightarrow0, 
\end{equation}
thus we see that for $\bfa, \bfa^\prime\in\ZZ^n$, $T(\bfa)\cong T(\bfa^\prime)$ if and only if there exists ${\bf y}\in\sfM$ such that $a_i=a_i^\prime+\sigma_i({\bf y})$ for all $i=1, \cdots, n$ 
(see e.g., \cite[Corollary~4.56]{BG2}). 
Let $\fkp_i\coloneqq T(0,\cdots,0,\overset{i}{\check{1}},0,\cdots,0)$, and consider the prime divisor $\calD_i\coloneqq \rmV(\fkp_i)$ on $\Spec R$. 
Then, a divisorial ideal $T(\bfa)=T(a_1,\cdots,a_n)$ corresponds to the Weil divisor $-(a_1\calD_1+\cdots +a_n\calD_n)$, and the exact sequence (\ref{cl_seq}) gives the relation: 
\begin{equation}
\label{relation_divisor}
v_{1,j}\calD_1+\cdots+v_{n,j}\calD_n=0, 
\end{equation}
for all $j=1,\cdots,d$, where $\bfv_i={}^t(v_{i,1},\cdots,v_{i,d})\in\ZZ^d$ for $i=1,\cdots,n$. 

We then define one of the nice class of divisorial ideals, called \emph{conic}. 
We say that $T(\bfa)$ is a \emph{conic module} (or \emph{conic divisorial ideal}) if there exists ${\bf x}\in\sfM_\RR$ such that $\bfa=\ulcorner\sigma({\bf x})\urcorner$. 
Conic modules are also characterized as $R$-modules appearing in $R^{1/m}$ as direct summands for $m\gg 0$ \cite[Proposition~3.6]{BG1}, thus these are related with positive characteristic commutative algebra (see e.g., \cite{Bru}). 
Also, we see that a conic module $T(\bfa)$ is an MCM module. 
Since the number of rank one MCM $R$-modules is finite up to isomorphism \cite[Corollary~5.2]{BG1}, we have only finitely many non-isomorphic conic modules. We note that in general there exists a divisorial ideal that is an MCM module but not conic. 
In fact, we will see such an example in Section~\ref{sec_MCM_HibiZ2}. 
In contrast to it, we see that any rank one MCM module is conic if $\Cl(R)\cong\ZZ$ (see Proposition~\ref{MCM_conic}). 

\medskip

A toric ring $R$ defined above can be described as the ring of invariants under the action of $G\coloneqq\Hom(\Cl(R),\kk^\times)\cong(\kk^\times)^r$ on $S\coloneqq \kk[x_1,\cdots,x_n]$.  
Let $\rmX(G)$ be the character group of $G$, which is isomorphic to $\Cl(R)$. 
(By the abuse of notations, we will use the same symbol for both of a character and the corresponding weight.) 
When we consider the prime divisor $\calD_i$ on $\Spec R$ as the element in $\rmX(G)\cong\Cl(R)$ via the surjection in (\ref{cl_seq}), we denote it by $\beta_i$. 
For a character $\chi\in\rmX(G)$, we denote by $V_\chi$ the irreducible representation corresponding to $\chi$, and we let $W=\bigoplus_{i=1}^nV_{\beta_i}$. 
Then, the symmetric algebra $\Sym W$ of the $G$-representation $W$ is isomorphic to $S$, and the algebraic torus $G$ acts on $S$, 
that is, $g\in G$ acts on $x_i$ as $g\cdot x_i=\beta_i(g)x_i$. 
Here, we recall the notion of the \emph{generic action}, and the above action of $G$ on $S$ is generic (see \cite[Subsection~10.6]{SpVdB}). 

\begin{definition}[{see \cite[Definition~1.6]{SpVdB}}]
\label{def_generic}
We say an action of $G$ on $S$ is \emph{generic} if $\Spec S$ contains a point with closed $G$-orbit and trivial stabilizer, and the subset of such points has codimension at least two in $\Spec S$. 
\end{definition}

Then, this action gives the $\Cl(R)$-grading on $S$, and the degree zero part coincides with the $G$-invariant components. In particular, we have that $R=S^G$ (see e.g., \cite[Theorem~2.1]{BG1}). 
Also, we say that $W$ is \emph{quasi-symmetric} if for every line $\ell\subset\rmX(G)_\RR\coloneqq\rmX(G)\otimes_\ZZ\RR$ passing through the origin, 
we have $\sum_{\beta_i\in\ell}\beta_i=0$ (see \cite[Subsection~1.6]{SpVdB}). We also say that a toric ring $R$ is \emph{quasi-symmetric} if $R\cong S^G$ with $S=\Sym W$ and $W$ is a quasi-symmetric representation. 
If $W$ is quasi-symmetric, then the top exterior $\bigwedge^nW$ is the trivial representation, and hence $R=S^G$ is Gorenstein. 
For a character $\chi$, we call an $R$-module with the form $M_{\chi}\coloneqq(S\otimes_\kk V_{\chi})^G$ \emph{module of covariants}, which is generated by $f\in S$ with $g\cdot f=\chi(g)f$ for any $g\in G$. 
In particular, for $\chi=\sum_ia_i\beta_i\in\rmX(G)$ we have that $T(a_1,\cdots, a_n)=M_{-\chi}$. 
Furthermore, by the arguments in \cite[Section~10.6]{SpVdB}, we see that $M_{-\chi}=T(a_1,\cdots, a_n)$ is conic if and only if $\chi$ is a strongly critical character. 
Here, we say that a character $\chi\in\rmX(G)$ is \emph{strongly critical} with respect to $\beta_1,\cdots,\beta_n$ if 
$\chi=\sum_ia_i\beta_i$ in $\rmX(G)_\RR$ with $a_i\in(-1,0]$ for all $i$. 

\subsection{Preliminaries on Hibi rings}
\label{subsec_Hibi}

In this subsection, we will consider a Hibi ring, which is a toric ring arising from a partially ordered set (= poset). 
Let $P=\{p_1,\cdots,p_{d-1}\}$ be a finite poset equipped with a partial order $\preceq$. 
For $p_i, p_j \in P$, we say that $p_i$ {\em covers} $p_j$ 
if $p_j \prec p_i$ and there is no $p' \in P$ such that $p_j \prec p' \prec p_i$ and $p'\neq p_i,p_j$. 
We then set $\widehat{P}=P \cup \{\hat{0}, \hat{1}\}$, where $\hat{0}$ (resp. $\hat{1}$) is the unique minimal (resp. maximal) element not belonging to the  poset $P$. 
We sometimes denote them as $p_0=\hat{0}$ and $p_d=\hat{1}$. 
The \emph{Hasse diagram}  $\calH(\widehat{P})$ of $\widehat{P}$ is a graph whose vertices are $\{p_0,p_1,\cdots,p_d\}$, and we draw an edge between $p_i$ and $p_j$ if $p_i$ covers $p_j$ or $p_j$ covers $p_i$. 
Thus, we say that $e=\{p_i,p_j\}$, where $0 \leq i \not= j \leq d$, is an {\em edge} of $\widehat{P}$ if $e$ is an edge of $\calH(\widehat{P})$. 
We say that a sequence $C=(p_{k_1},\cdots,p_{k_m})$ is a {\em cycle} in $\widehat{P}$ if $C$ forms a cycle 
in $\calH(\widehat{P})$, i.e., $p_{k_i} \not= p_{k_j}$ for $1 \leq i \not= j \leq m$ 
and each $\{p_{k_i},p_{k_{i+1}}\}$ is an edge of $\widehat{P}$ for $1 \leq i \leq m$, where $p_{k_{m+1}}=p_{k_1}$. 
Moreover, a cycle $C$ is said to be a {\em circuit} if $\{p_{k_i},p_{k_j}\}$ is not an edge of $\widehat{P}$ for any $1 \leq i,j \leq m$ with $|i-j| \geq 2$. 

For each edge $e=\{p_i,p_j\}$ of $\widehat{P}$ with $p_i \prec p_j$, 
let $\sigma_e$ be a linear form in $\RR^d$ defined by 
\begin{align*}
\sigma_e({\bf x})\coloneqq
\begin{cases}
x_i-x_j, \;&\text{ if }j \not= d, \\
x_i, \; &\text{ if }j=d 
\end{cases}
\end{align*}
for ${\bf x}=(x_0,x_1,\cdots,x_{d-1})$. Let $\tau_P=\mathrm{Cone}(\sigma_e \mid e \text{ is an edge of }\widehat{P}) \subset \sfN_\RR=\RR^d$. 
Then, the toric ring $\kk[P]\coloneqq\kk[\tau_P^\vee \cap \ZZ^d]$ is called the \emph{Hibi ring} associated with a poset $P$. 
The basic properties of $\kk[P]$ were studied in \cite{Hibi}. 
In particular, it is known that $\kk[P]$ is a CM normal domain with $\dim \kk[P]=|P|+1=d$. Moreover, $\kk[P]$ is Gorenstein if and only if $P$ is pure. 
Here, we say that $P$ is \emph{pure} if all of the maximal chains $p_{i_1} \prec \cdots \prec p_{i_\ell}$ have the same length. 

\begin{remark}
\label{remark_Hibi}
If there exists an edge $e$ of $\widehat{P}$ such that every maximal chain of $\widehat{P}$ contains $e$, then we easily see that the associated Hibi ring is the polynomial extension of a certain Hibi ring. 
For avoiding triviality, we assume that $\widehat{P}$ does not contain such an edge, and hence any Hibi ring is not the polynomial extension. 
\end{remark}

Next, we consider the divisor class group $\Cl(\kk[P])$ of a Hibi ring $\kk[P]$. 
Let $e_1,\cdots,e_n$ be all the edges of $\widehat{P}$. We set the linear form $\sigma:\RR^d \rightarrow \RR^n$ as 
$\sigma({\bf x})=(\sigma_{e_1}({\bf x}),\cdots,\sigma_{e_n}({\bf x}))$ where ${\bf x} \in \RR^d$. 
For $p \in \widehat{P} {\setminus} \{\hat{1}\}$, let $U(p)$ denote the set of all elements in $\widehat{P}$ which cover $p$. 
Also, for $p \in \widehat{P} {\setminus} \{\hat{0}\}$, let $D(p)$ denote the set of all elements in $\widehat{P}$ which are covered by $p$. 
By the construction of the cone $\tau_P$, each prime divisor is indexed by an edge of $\widehat{P}$.  
Using the description of $\sigma(-)$ and (\ref{relation_divisor}), we see that each prime divisor $\calD_e$ corresponding to an edge $e$ of $\widehat{P}$ satisfies the relations: 
\begin{align}
\label{relation_div_hibi}
\sum_{q \in U(p)} \calD_{\{q,p\}}=\sum_{q' \in D(p)} \calD_{\{p,q'\}} \text{ for }p \in \widehat{P} {\setminus} \{\hat{0}, \hat{1}\}, 
\; \text{ and }\; \sum_{q \in U(\hat{0})}\calD_{\{q,p_0\}}=0. 
\end{align}
In particular, using these relations, we see that the class group is a free abelian group as follows. 

\begin{theorem}[{see \cite{HHN}}] 
\label{classgroup_Hibi}
Let the notation be the same as above. 
Then, we have that $\Cl(\kk[P]) \cong  \ZZ^{n-d}$. 
Moreover, generators of $\Cl(\kk[P])$ can be given by prime divisors corresponding to edges of $\widehat{P}$ not contained in a spanning tree of $\widehat{P}$. 
\end{theorem}

Here, a {\em spanning tree} of $\widehat{P}$ is a set of $d$ edges $e_{i_1},\cdots,e_{i_d}$ of $\widehat{P}$ satisfying the conditions: 
\begin{itemize}
\item any element in $\widehat{P}$ is an endpoint of some edge $e_{i_j}$, 
\item the edges $e_{i_1},\cdots,e_{i_d}$ do not form cycles.
\end{itemize}
We remark that a spanning tree is not unique. 
Furthermore, conic classes in $\Cl(\kk[P]) \cong  \ZZ^{n-d}$ can be described as follows. 

\begin{theorem}[{see \cite[Theorem~2.4]{HN}}]
\label{conic_Hibi}
Let $e_1,\cdots,e_d$ be a spanning tree of $\widehat{P}$ and $e_{d+1},\cdots,e_n$ be the remaining edges of $\widehat{P}$. 
For a circuit $C=(p_{k_1},\cdots,p_{k_m})$ in $\widehat{P}$, we define the subsets of edges 
\begin{align*}
X_C^+&=\{\{p_{k_i},p_{k_{i+1}}\} \mid 1 \leq i \leq m, \;  p_{k_i} \prec p_{k_{i+1}}\}, \\
X_C^-&=\{\{p_{k_i},p_{k_{i+1}}\} \mid 1 \leq i \leq m, \; p_{k_{i+1}} \prec p_{k_i}\}, \\
Y_C^{\pm}&=X_C^{\pm} \cap \{e_1,\cdots,e_d\}, \text{ and }\; 
Z_C^{\pm} = X_C^\pm \cap \{e_{d+1},\cdots,e_n\}, 
\end{align*}
where $p_{k_{m+1}}=p_{k_1}$. We define the convex polytope 
\begin{align}\label{ccp}
\calC(P)=\left\{(z_1,\cdots,z_{n-d}) \in \RR^{n-d} \mid 
-|X_C^-|+1 \leq \sum_{e_{d+\ell} \in Z_C^+}  z_\ell - \sum_{e_{d+\ell'} \in Z_C^-}  z_{\ell'} \leq |X_C^+|-1\right\}, 
\end{align}
where $C=(p_{k_1},\cdots,p_{k_m})$ runs over all circuits in $\widehat{P}$. 
Then, conic modules of $\kk[P]$ precisely correspond to points in $\calC(P) \cap \Cl(\kk[P])=\calC(P) \cap \ZZ^{n-d}$. 
\end{theorem}

As we mentioned, conic modules are rank one MCM modules, but there exists a rank one MCM module that is not conic. 
In Section~\ref{sec_MCM_HibiZ2}, we will determine all rank one MCM modules over Hibi rings with class group $\ZZ^2$, and we can find many examples of non-conic rank one MCM modules there. 

\begin{example}
We consider the poset $P$ whose Hasse diagram takes the following form. 

\medskip
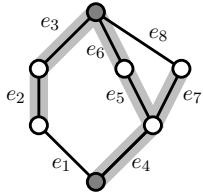
\begin{center}
{\scalebox{0.75}{
\begin{tikzpicture}

\coordinate (Min) at (0,0); \coordinate (Max) at (0,3);
\coordinate (P1) at (-1,1); \coordinate (P2) at (-1,2); \coordinate (P3) at (1,1);
\coordinate (P4) at (0.5,2); \coordinate (P5) at (1.5,2); 
\draw[line width=0.32cm, lightgray]  (P1)--(P2)--(Max) ; \draw[line width=0.32cm, lightgray]  (Min)--(P3)--(P4)--(Max) ; 
\draw[line width=0.32cm, lightgray]  (Min)--(P3)--(P5) ; 
\draw[line width=0.05cm]  (Min)--(P1)--(P2)--(Max); \draw[line width=0.05cm]  (Min)--(P3)--(P4)--(Max); \draw[line width=0.05cm]  (Min)--(P3)--(P5)--(Max); 

\draw [line width=0.05cm, fill=gray] (Min) circle [radius=0.15] ; \draw [line width=0.05cm, fill=gray] (Max) circle [radius=0.15] ; 
\draw [line width=0.05cm, fill=white] (P1) circle [radius=0.15] ; \draw [line width=0.05cm, fill=white] (P2) circle [radius=0.15] ;
\draw [line width=0.05cm, fill=white] (P3) circle [radius=0.15] ;
\draw [line width=0.05cm, fill=white] (P4) circle [radius=0.15] ; \draw [line width=0.05cm, fill=white] (P5) circle [radius=0.15] ;

\node at (-0.6,0.3) {\large $e_1$}; \node at (-1.4,1.5) {\large $e_2$}; \node at (-0.8,2.8) {\large $e_3$}; 
\node at (0.8,0.3) {\large $e_4$}; \node at (0.35,1.5) {\large $e_5$}; \node at (0,2.3) {\large $e_6$}; 
\node at (1.7,1.5) {\large $e_7$}; \node at (1.1,2.6) {\large $e_8$}; 
\end{tikzpicture}
} }
\end{center}
\medskip

In this case, we see that the set of grayed edges $\{e_2,e_3,\cdots,e_7\}$ is a spanning tree, thus the prime divisors $\calD_{e_1},\calD_{e_8}$ generate $\Cl(R)$, which is isomorphic to $\ZZ^2$. 
Using the relations (\ref{relation_div_hibi}), we have that 
$$
\calD_{e_1}=\calD_{e_2}=\calD_{e_3}=-\calD_{e_4}, \quad \calD_{e_5}=\calD_{e_6}=-\calD_{e_1}-\calD_{e_8}, \quad \calD_{e_7}=\calD_{e_8}. 
$$
Furthermore, by Theorem~\ref{conic_Hibi} the convex polytope $\calC(P)$ can be described as 
$$
\calC(P)=\{(z_1,z_8)\in\RR^2\mid -2\le z_1\le2,\, -1\le z_8\le1,\, -2\le z_1-z_8\le2\}. 
$$
Thus, divisorial ideals corresponding to $a_1\calD_{e_1}+a_8\calD_{e_8}$ with $(a_1,a_8)\in\calC(P)\cap\ZZ^2$ are conic modules. 
\end{example}

\section{Splitting NCCRs for Hibi rings with class group $\ZZ^2$}
\label{sec_HibiZ2}

In this section, we study splitting NCCRs for some Hibi rings. Since the class group of a Hibi ring is free abelian, non-Gorenstein Hibi rings do not admit an NCCR (see Remark~\ref{rem_NCCR}(b)). 
Thus, we will consider Gorenstein Hibi rings, which are coming from pure posets. 
First, since Hibi rings are toric rings, Gorenstein Hibi rings with class group $\ZZ$ have a splitting NCCR (see Theorem~\ref{NCCR_Z}). 
In particular, we easily have the following. 

\begin{example}
Let $R$ be a Gorenstein Hibi ring with $\Cl(R)\cong\ZZ$. We assume that $R$ is not a polynomial extension (see Remark~\ref{remark_Hibi}).
Then, the Hasse diagram $\calH(\widehat{P})$ of a poset $\widehat{P}$ giving such a Hibi ring $R$ takes the following form. 

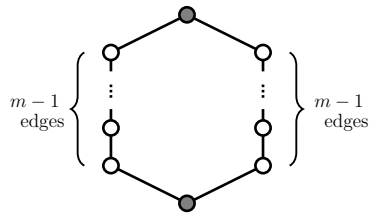
\begin{center}
{\scalebox{0.5}{
\begin{tikzpicture}

\coordinate (Min) at (2,-1); \coordinate (Max) at (2,4);
\coordinate (N11) at (0,0); \coordinate (N12) at (0,1); \coordinate (N15) at (0,3);
\coordinate (Nt1) at (4,0); \coordinate (Nt2) at (4,1); \coordinate (Nt5) at (4,3);
\draw[line width=0.07cm]  (N11)--(N12); \draw[line width=0.07cm]  (N12)--(0,1.5); \draw[line width=0.07cm]  (0,2.5)--(N15); 
\draw[line width=0.07cm,dotted]  (0,1.8)--(0,2.2);
\draw[line width=0.07cm]  (Nt1)--(Nt2); \draw[line width=0.07cm]  (Nt2)--(4,1.5); \draw[line width=0.07cm]  (4,2.5)--(Nt5); 
\draw[line width=0.07cm,dotted]  (4,1.8)--(4,2.2);
\draw[line width=0.07cm]  (N11)--(Min); \draw[line width=0.07cm]  (Nt1)--(Min); 
\draw[line width=0.07cm]  (N15)--(Max); \draw[line width=0.07cm]  (Nt5)--(Max); 

\draw [line width=0.07cm, fill=gray] (Min) circle [radius=0.2] ; \draw [line width=0.07cm, fill=gray] (Max) circle [radius=0.2] ; 
\draw [line width=0.07cm, fill=white] (N11) circle [radius=0.2] ; \draw [line width=0.07cm, fill=white] (N12) circle [radius=0.2] ;
\draw [line width=0.07cm, fill=white] (N15) circle [radius=0.2] ;
\draw [line width=0.07cm, fill=white] (Nt1) circle [radius=0.2] ; \draw [line width=0.07cm, fill=white] (Nt2) circle [radius=0.2] ;
\draw [line width=0.07cm, fill=white] (Nt5) circle [radius=0.2] ;

\draw [line width=0.05cm, decorate,decoration={brace,amplitude=10pt}](-0.7,0) -- (-0.7,3) node[black,midway,xshift=-1.3cm,yshift=0.2cm] {\LARGE $m-1$}
node[black,midway,xshift=-1.1cm,yshift=-0.35cm] {\LARGE edges}; 
\draw [line width=0.05cm, decorate,decoration={brace,mirror,amplitude=10pt}](4.7,0) -- (4.7,3) node[black,midway,xshift=1.3cm,yshift=0.2cm] {\LARGE $m-1$}
node[black,midway,xshift=1.5cm,yshift=-0.35cm] {\LARGE edges}; 
\end{tikzpicture}
} }
\end{center}
Such a Hibi ring $R$ is isomorphic to the Segre product of two polynomial rings with $m+1$ variables (see e.g., \cite[Example~2.6]{HN}), 
thus we have that $$R\cong\kk[x_0,\cdots,x_m]\#\kk[y_0,\cdots,y_m]=\kk[x_iy_j \mid i,j=0,\cdots,m].$$ 
Let $T(a)$ be a divisorial ideal corresponding to $a\in\Cl(R)$. 
Then any conic module is described as $T(a)$ with $a\in[-m,m]\cap\ZZ\subset\Cl(R)$ (see Theorem~\ref{conic_Hibi}), 
and any rank one MCM module is conic (see Lemma~\ref{MCM_conic}). 
By Theorem~\ref{NCCR_Z}, $N\coloneqq T(0)\oplus\cdots\oplus T(m)$ gives a splitting NCCR of $R$, and hence any module giving a splitting NCCR takes the form $$(N\otimes_RT(b))^{**}\cong T(b)\oplus T(b+1)\oplus\cdots\oplus T(b+m)$$ with $b\in\ZZ$ by Proposition~\ref{NCCR_Z_any}. 
\end{example}

Accordingly, we then consider the case where $R$ is a Gorenstein Hibi ring with $\Cl(R)\cong\ZZ^2$. First, we classify the posets giving such a Hibi ring. 
We note that even if we  turn $\calH(\widehat{P})$ upside down, it gives the same Hibi ring up to isomorphism. 
Thus, we classify the desired posets up to this operation. 

\begin{lemma}
\label{poset_Z2}
Let $R=\kk[P]$ be the Gorenstein Hibi ring associated with a pure poset $P$. 
We assume that $\Cl(R)\cong\ZZ^2$ and $R$ is not a polynomial extension of a Hibi ring (see Remark~\ref{remark_Hibi}). 
Then, the Hasse diagram $\calH(\widehat{P})$ looks like one of posets {\rm(I)--(V)} shown in Figure~\ref{figure_poset_Z2} where grayed vertices are $\hat{0}, \hat{1}$ respectively. 
We remark that since $R$ is Gorenstein each poset $P$ is pure, that is, all of the maximal chains have the same length. 
We refer to Figure~\ref{poset_typeI}--\ref{poset_typeV} for more precise lengths of chains. 
\end{lemma}

\begin{figure}[h!]
\begin{center}
\newcommand{\edgewidth}{0.11cm} 
\newcommand{\nodewidth}{0.11cm} 
\newcommand{\noderad}{0.28} 

\begin{tikzpicture}
\node at (0,-2.3) {(I)} ; 
\node at (0,0) 
{\scalebox{0.25}{
\begin{tikzpicture}

\coordinate (Min) at (4,0); \coordinate (Max) at (4,15);
\coordinate (N11) at (0,2); \coordinate (N12) at (0,7); \coordinate (N13) at (0,8); \coordinate (N14) at (0,13); 
\coordinate (N21) at (7,2); \coordinate (N22) at (7,7); 
\coordinate (N23a) at (6,8);  \coordinate (N24a) at (6,13); \coordinate (N23b) at (8,8); \coordinate (N24b) at (8,13); 

\draw[line width=\edgewidth]  (Min)--(N11); \draw[line width=\edgewidth]  (N11)--(0,2.8); 
\draw[line width=\edgewidth, loosely dotted]  (0,3.4)--(0,5.6); \draw[line width=\edgewidth]  (0,6.2)--(N12); 
\draw[line width=\edgewidth]  (N12)--(N13); \draw[line width=\edgewidth]  (N13)--(0,8.8); 
\draw[line width=\edgewidth, loosely dotted]  (0,9.4)--(0,11.6); \draw[line width=\edgewidth]  (0,12.2)--(N14); \draw[line width=\edgewidth]  (N14)--(Max); 

\draw[line width=\edgewidth]  (Min)--(N21); \draw[line width=\edgewidth]  (N21)--(7,2.8); 
\draw[line width=\edgewidth, loosely dotted]  (7,3.4)--(7,5.6); \draw[line width=\edgewidth]  (7,6.2)--(N22); 

\draw[line width=\edgewidth]  (N22)--(N23a); \draw[line width=\edgewidth]  (N22)--(N23b); 
\draw[line width=\edgewidth]  (N23a)--(6,8.8); \draw[line width=\edgewidth]  (N23b)--(8,8.8); 
\draw[line width=\edgewidth, loosely dotted]  (6,9.4)--(6,11.6); \draw[line width=\edgewidth, loosely dotted]  (8,9.4)--(8,11.6); 
\draw[line width=\edgewidth]  (6,12.2)--(N24a); \draw[line width=\edgewidth]  (N24a)--(Max); 
\draw[line width=\edgewidth]  (8,12.2)--(N24b); \draw[line width=\edgewidth]  (N24b)--(Max); 

\draw [line width=\nodewidth, fill=gray] (Min) circle [radius=\noderad] ; \draw [line width=\nodewidth, fill=gray] (Max) circle [radius=\noderad] ; 
\draw [line width=\nodewidth, fill=white] (N11) circle [radius=\noderad] ; \draw [line width=\nodewidth, fill=white] (N12) circle [radius=\noderad] ; 
\draw [line width=\nodewidth, fill=white] (N13) circle [radius=\noderad] ; \draw [line width=\nodewidth, fill=white] (N14) circle [radius=\noderad] ; 
\draw [line width=\nodewidth, fill=white] (N21) circle [radius=\noderad] ; \draw [line width=\nodewidth, fill=white] (N22) circle [radius=\noderad] ; 
\draw [line width=\nodewidth, fill=white] (N23a) circle [radius=\noderad] ; \draw [line width=\nodewidth, fill=white] (N24a) circle [radius=\noderad] ; 
\draw [line width=\nodewidth, fill=white] (N23b) circle [radius=\noderad] ; \draw [line width=\nodewidth, fill=white] (N24b) circle [radius=\noderad] ; 

\end{tikzpicture}
} } ;

\node at (4.6,-2.3) {(II)} ; 
\node at (4.6,0) 
{\scalebox{0.25}{
\begin{tikzpicture}

\coordinate (Min) at (4,0); \coordinate (Max) at (4,15);
\coordinate (N11) at (0,2); \coordinate (N12) at (0,5); \coordinate (N13) at (0,6); \coordinate (N14) at (0,9); \coordinate (N15) at (0,10); \coordinate (N16) at (0,13); 
\coordinate (N21) at (8,2); \coordinate (N22) at (8,5); \coordinate (N23) at (8,6); \coordinate (N24) at (8,9); \coordinate (N25) at (8,10); \coordinate (N26) at (8,13); 
\coordinate (N31) at (2,8.75); \coordinate (N32) at (6,6.25);

\draw[line width=\edgewidth]  (Min)--(N11); \draw[line width=\edgewidth]  (N11)--(0,2.8); \draw[line width=\edgewidth, loosely dotted]  (0,3)--(0,4); 
\draw[line width=\edgewidth]  (0,4.2)--(N12); \draw[line width=\edgewidth]  (N12)--(N13); \draw[line width=\edgewidth]  (N13)--(0,6.8); 
\draw[line width=\edgewidth, loosely dotted]  (0,7)--(0,8); \draw[line width=\edgewidth]  (0,8.2)--(N14); \draw[line width=\edgewidth]  (N14)--(N15); 
\draw[line width=\edgewidth]  (N15)--(0,10.8); \draw[line width=\edgewidth, loosely dotted]  (0,11)--(0,12); 
\draw[line width=\edgewidth]  (0,12.2)--(N16); \draw[line width=\edgewidth]  (N16)--(Max); 

\draw[line width=\edgewidth]  (Min)--(N21); \draw[line width=\edgewidth]  (N21)--(8,2.8); \draw[line width=\edgewidth, loosely dotted]  (8,3)--(8,4); 
\draw[line width=\edgewidth]  (8,4.2)--(N22); \draw[line width=\edgewidth]  (N22)--(N23); \draw[line width=\edgewidth]  (N23)--(8,6.8); 
\draw[line width=\edgewidth, loosely dotted]  (8,7)--(8,8); \draw[line width=\edgewidth]  (8,8.2)--(N24); \draw[line width=\edgewidth]  (N24)--(N25); 
\draw[line width=\edgewidth]  (N25)--(8,10.8); \draw[line width=\edgewidth, loosely dotted]  (8,11)--(8,12); 
\draw[line width=\edgewidth]  (8,12.2)--(N26); \draw[line width=\edgewidth]  (N26)--(Max); 

\draw[line width=\edgewidth]  (N15)--(N31); \draw[line width=\edgewidth]  (N31)--(2.8,8.25); 
\draw[line width=\edgewidth, loosely dotted]  (3.2,8)--(4.8,7); 
\draw[line width=\edgewidth]  (5.2,6.75)--(N32); \draw[line width=\edgewidth]  (N32)--(N22); 

\draw [line width=\nodewidth, fill=gray] (Min) circle [radius=\noderad] ; \draw [line width=\nodewidth, fill=gray] (Max) circle [radius=\noderad] ; 
\draw [line width=\nodewidth, fill=white] (N11) circle [radius=\noderad] ; \draw [line width=\nodewidth, fill=white] (N12) circle [radius=\noderad] ; \draw [line width=\nodewidth, fill=white] (N13) circle [radius=\noderad] ; 
\draw [line width=\nodewidth, fill=white] (N14) circle [radius=\noderad] ; \draw [line width=\nodewidth, fill=white] (N15) circle [radius=\noderad] ; \draw [line width=\nodewidth, fill=white] (N16) circle [radius=\noderad] ; 
\draw [line width=\nodewidth, fill=white] (N21) circle [radius=\noderad] ; \draw [line width=\nodewidth, fill=white] (N22) circle [radius=\noderad] ; \draw [line width=\nodewidth, fill=white] (N23) circle [radius=\noderad] ; 
\draw [line width=\nodewidth, fill=white] (N24) circle [radius=\noderad] ; \draw [line width=\nodewidth, fill=white] (N25) circle [radius=\noderad] ; \draw [line width=\nodewidth, fill=white] (N26) circle [radius=\noderad] ; 
\draw [line width=\nodewidth, fill=white] (N31) circle [radius=\noderad] ; \draw [line width=\nodewidth, fill=white] (N32) circle [radius=\noderad] ; 

\end{tikzpicture}
} } ;

\node at (9.2,-2.3) {(III)} ; 
\node at (9.2,0) 
{\scalebox{0.25}{
\begin{tikzpicture}

\coordinate (Min) at (4,0); \coordinate (Max) at (4,15);
\coordinate (N11) at (0,2); \coordinate (N12) at (0,5); \coordinate (N13) at (0,6); 
\coordinate (N14) at (0,9); \coordinate (N15) at (0,10); \coordinate (N16) at (0,13); 

\coordinate (N21) at (7,2); \coordinate (N22) at (7,5); 
\coordinate (N23a) at (6,6);  \coordinate (N24a) at (6,9); \coordinate (N23b) at (8,6); \coordinate (N24b) at (8,9); 
\coordinate (N25) at (7,10); \coordinate (N26) at (7,13); 

\draw[line width=\edgewidth]  (Min)--(N11); \draw[line width=\edgewidth]  (N11)--(0,2.8); 
\draw[line width=\edgewidth, loosely dotted]  (0,3)--(0,4); 
\draw[line width=\edgewidth]  (0,4.2)--(N12); \draw[line width=\edgewidth]  (N12)--(N13); \draw[line width=\edgewidth]  (N13)--(0,6.8); 
\draw[line width=\edgewidth, loosely dotted]  (0,7)--(0,8); \draw[line width=\edgewidth]  (0,8.2)--(N14); 
\draw[line width=\edgewidth]  (N14)--(N15); \draw[line width=\edgewidth]  (N15)--(0,10.8); 
\draw[line width=\edgewidth, loosely dotted]  (0,11)--(0,12); \draw[line width=\edgewidth]  (0,12.2)--(N16); \draw[line width=\edgewidth]  (N16)--(Max); 

\draw[line width=\edgewidth]  (Min)--(N21); \draw[line width=\edgewidth]  (N21)--(7,2.8); 
\draw[line width=\edgewidth, loosely dotted]  (7,3)--(7,4); 
\draw[line width=\edgewidth]  (7,4.2)--(N22); \draw[line width=\edgewidth]  (N22)--(N23a); \draw[line width=\edgewidth]  (N22)--(N23b); 
\draw[line width=\edgewidth]  (N23a)--(6,6.8); \draw[line width=\edgewidth]  (N23b)--(8,6.8); 
\draw[line width=\edgewidth, loosely dotted]  (6,7)--(6,8); \draw[line width=\edgewidth, loosely dotted]  (8,7)--(8,8); 
\draw[line width=\edgewidth]  (6,8.2)--(N24a); \draw[line width=\edgewidth]  (8,8.2)--(N24b); 
\draw[line width=\edgewidth]  (N24a)--(N25); \draw[line width=\edgewidth]  (N24b)--(N25); \draw[line width=\edgewidth]  (N25)--(7,10.8); 
\draw[line width=\edgewidth, loosely dotted]  (7,11)--(7,12); \draw[line width=\edgewidth]  (7,12.2)--(N26); \draw[line width=\edgewidth]  (N26)--(Max); 

\draw [line width=\nodewidth, fill=gray] (Min) circle [radius=\noderad] ; \draw [line width=\nodewidth, fill=gray] (Max) circle [radius=\noderad] ; 
\draw [line width=\nodewidth, fill=white] (N11) circle [radius=\noderad] ; \draw [line width=\nodewidth, fill=white] (N12) circle [radius=\noderad] ; 
\draw [line width=\nodewidth, fill=white] (N13) circle [radius=\noderad] ; \draw [line width=\nodewidth, fill=white] (N14) circle [radius=\noderad] ; 
\draw [line width=\nodewidth, fill=white] (N15) circle [radius=\noderad] ; \draw [line width=\nodewidth, fill=white] (N16) circle [radius=\noderad] ; 
\draw [line width=\nodewidth, fill=white] (N21) circle [radius=\noderad] ; \draw [line width=\nodewidth, fill=white] (N22) circle [radius=\noderad] ; 
\draw [line width=\nodewidth, fill=white] (N23a) circle [radius=\noderad] ; \draw [line width=\nodewidth, fill=white] (N24a) circle [radius=\noderad] ; 
\draw [line width=\nodewidth, fill=white] (N23b) circle [radius=\noderad] ; \draw [line width=\nodewidth, fill=white] (N24b) circle [radius=\noderad] ; 
\draw [line width=\nodewidth, fill=white] (N25) circle [radius=\noderad] ; \draw [line width=\nodewidth, fill=white] (N26) circle [radius=\noderad] ; 

\end{tikzpicture}
} } ;

\node at (2.3,-6.3) {(IV)} ; 
\node at (2.3,-4) 
{\scalebox{0.25}{
\begin{tikzpicture}
\coordinate (Min) at (4,0); \coordinate (Max) at (4,14);
\coordinate (N11) at (0,2); \coordinate (N12) at (0,5); \coordinate (N13) at (0,9); \coordinate (N14) at (0,12); 
\coordinate (N21) at (8,2); \coordinate (N22) at (8,5); \coordinate (N23) at (8,9); \coordinate (N24) at (8,12); 
\coordinate (N3) at (4,7);

\draw[line width=\edgewidth]  (Min)--(N11); \draw[line width=\edgewidth]  (N11)--(0,2.8); \draw[line width=\edgewidth, loosely dotted]  (0,3)--(0,4); 
\draw[line width=\edgewidth]  (0,4.2)--(N12); \draw[line width=\edgewidth]  (N12)--(N3); \draw[line width=\edgewidth]  (N3)--(N13); 
\draw[line width=\edgewidth]  (N13)--(0,9.8); \draw[line width=\edgewidth, loosely dotted]  (0,10)--(0,11); 
\draw[line width=\edgewidth]  (0,11.2)--(N14); \draw[line width=\edgewidth]  (N14)--(Max);
\draw[line width=\edgewidth]  (Min)--(N21); \draw[line width=\edgewidth]  (N21)--(8,2.8); \draw[line width=\edgewidth, loosely dotted]  (8,3)--(8,4); 
\draw[line width=\edgewidth]  (8,4.2)--(N22); \draw[line width=\edgewidth]  (N22)--(N3); \draw[line width=\edgewidth]  (N3)--(N23); 
\draw[line width=\edgewidth]  (N23)--(8,9.8); \draw[line width=\edgewidth, loosely dotted]  (8,10)--(8,11); 
\draw[line width=\edgewidth]  (8,11.2)--(N24); \draw[line width=\edgewidth]  (N24)--(Max);

\draw [line width=\nodewidth, fill=gray] (Min) circle [radius=\noderad] ; \draw [line width=\nodewidth, fill=gray] (Max) circle [radius=\noderad] ; 
\draw [line width=\nodewidth, fill=white] (N11) circle [radius=\noderad] ; \draw [line width=\nodewidth, fill=white] (N12) circle [radius=\noderad] ; 
\draw [line width=\nodewidth, fill=white] (N13) circle [radius=\noderad] ; \draw [line width=\nodewidth, fill=white] (N14) circle [radius=\noderad] ; 
\draw [line width=\nodewidth, fill=white] (N21) circle [radius=\noderad] ; \draw [line width=\nodewidth, fill=white] (N22) circle [radius=\noderad] ; 
\draw [line width=\nodewidth, fill=white] (N23) circle [radius=\noderad] ; \draw [line width=\nodewidth, fill=white] (N24) circle [radius=\noderad] ; 
\draw [line width=\nodewidth, fill=white] (N3) circle [radius=\noderad] ; 
\end{tikzpicture}
} } ;

\node at (6.9,-6.3) {(V)} ; 
\node at (6.9,-4) 
{\scalebox{0.25}{
\begin{tikzpicture}
\coordinate (Min) at (4,0); \coordinate (Max) at (4,14);
\coordinate (N11) at (0,2); \coordinate (N12) at (0,3); \coordinate (N13) at (0,11); \coordinate (N14) at (0,12); 
\coordinate (N21) at (4,2); \coordinate (N22) at (4,3); \coordinate (N23) at (4,11); \coordinate (N24) at (4,12); 
\coordinate (N31) at (8,2); \coordinate (N32) at (8,3); \coordinate (N33) at (8,11); \coordinate (N34) at (8,12); 

\draw[line width=\edgewidth]  (Min)--(N11); \draw[line width=\edgewidth]  (N11)--(N12); \draw[line width=\edgewidth]  (N12)--(0,3.8); 
\draw[line width=\edgewidth, loosely dotted]  (0,5)--(0,9); 
\draw[line width=\edgewidth]  (0,10.2)--(N13); \draw[line width=\edgewidth]  (N13)--(N14); \draw[line width=\edgewidth]  (N14)--(Max); 
\draw[line width=\edgewidth]  (Min)--(N21); \draw[line width=\edgewidth]  (N21)--(N22); \draw[line width=\edgewidth]  (N22)--(4,3.8); 
\draw[line width=\edgewidth, loosely dotted]  (4,5)--(4,9); 
\draw[line width=\edgewidth]  (4,10.2)--(N23); \draw[line width=\edgewidth]  (N23)--(N24); \draw[line width=\edgewidth]  (N24)--(Max); 
\draw[line width=\edgewidth]  (Min)--(N31); \draw[line width=\edgewidth]  (N31)--(N32); \draw[line width=\edgewidth]  (N32)--(8,3.8); 
\draw[line width=\edgewidth, loosely dotted]  (8,5)--(8,9); 
\draw[line width=\edgewidth]  (8,10.2)--(N33); \draw[line width=\edgewidth]  (N33)--(N34); \draw[line width=\edgewidth]  (N34)--(Max); 

\draw [line width=\nodewidth, fill=gray] (Min) circle [radius=\noderad] ; \draw [line width=\nodewidth, fill=gray] (Max) circle [radius=\noderad] ; 
\draw [line width=\nodewidth, fill=white] (N11) circle [radius=\noderad] ; \draw [line width=\nodewidth, fill=white] (N12) circle [radius=\noderad] ; 
\draw [line width=\nodewidth, fill=white] (N13) circle [radius=\noderad] ; \draw [line width=\nodewidth, fill=white] (N14) circle [radius=\noderad] ; 
\draw [line width=\nodewidth, fill=white] (N21) circle [radius=\noderad] ; \draw [line width=\nodewidth, fill=white] (N22) circle [radius=\noderad] ; 
\draw [line width=\nodewidth, fill=white] (N23) circle [radius=\noderad] ; \draw [line width=\nodewidth, fill=white] (N24) circle [radius=\noderad] ; 
\draw [line width=\nodewidth, fill=white] (N31) circle [radius=\noderad] ; \draw [line width=\nodewidth, fill=white] (N32) circle [radius=\noderad] ; 
\draw [line width=\nodewidth, fill=white] (N33) circle [radius=\noderad] ; \draw [line width=\nodewidth, fill=white] (N34) circle [radius=\noderad] ; 
\end{tikzpicture}
} } ;
\end{tikzpicture}
\end{center}
\caption{Hasse diagrams $\calH(\widehat{P})$ for Gorenstein Hibi rings $R$ with $\Cl(R)\cong\ZZ^2$}
\label{figure_poset_Z2}
\end{figure}

\begin{proof}
We consider $\calH(\widehat{P})$ as a simple graph, and denote it by $\calG$. 
Let $V(\calG)$ be the set of vertices of $\calG$, and $E(\calG)$ be the set of edges of $\calG$, especially $V(\calG)$ coincides with $\widehat{P}$. 
Since $\Cl(R)\cong\ZZ^2$, we have that $|E(\calG)|-|V(\calG)|=1$ (see Theorem~\ref{classgroup_Hibi}). 
For $v\in V(\calG)$, $\deg_\calG(v)$ denotes the \emph{degree} (or \emph{valency}) of $v$, which is the number of edges incident to $v$. 

By the construction of $\calH(\widehat{P})$, we see that $\calG$ is connected. Also, there are no vertices whose degree is $1$. 
In fact, since any vertex $v\in V(\calG){\setminus}\{\hat{0},\hat{1}\}=P$ satisfies $\hat{0}\prec v\prec\hat{1}$, 
we have that $\deg_\calG(v)\neq 1$. If $\deg_\calG(\hat{0})=1$ or $\deg_\calG(\hat{1})=1$, then the associated Hibi ring is obtained as the polynomial extension of a certain Hibi ring. 
Thus, we conclude that $\deg_\calG(v)\ge 2$ for any $v\in\calG$. 

We now describe all the possible connected graphs with $|E(\calG)|-|V(\calG)|=1$ and $\deg_\calG(v)\ge 2$ for any $v\in\calG$. 
By the handshaking lemma (see e.g., \cite[p.4]{Bol}), we see that $2|E(\calG)|=\sum_{v\in V(\calG)}\deg_\calG(v)$, and hence we have that 
$$
|V(\calG)|+1=\sum_{v\in V(\calG)}\frac{\deg_\calG(v)}{2}
$$
Moreover, since $\deg_\calG(v)\ge 2$ for any $v\in\calG$, only the following two situations may happen: 
\begin{itemize}
\item[(i)] $\calG$ has exactly two vertices with degree $3$ and all the others have degree $2$, or 
\item[(ii)] $\calG$ has a unique vertex with degree $4$ and all the others have degree $2$. 
\end{itemize}

\smallskip

\noindent{\bf The case (i)}: Let $v$ and $w$ be the vertices of $\calG$ with degree $3$. Then, $\deg_\calG u=2$ for any $u\in V(\calG){\setminus}\{v,w\}$. 
\begin{itemize}
\item Let $v=\hat{0}$ and $w=\hat{1}$. Then it is easy to see that $\calG$ looks like (V). 
\item Let $v=\hat{1}$ and $w\not=\hat{0}$, and hence $w\in V(\calG){\setminus}\{\hat{0},\hat{1}\}$. Since other vertices have degree $2$, it follows that $\calG$ looks like (I). 
Similarly, when $v=\hat{0}$ and $w\not=\hat{1}$, we obtain the graph that is the upside-down of (I), and it gives the same Hibi ring up to isomorphism. 
\item Let $v, w\in V(\calG){\setminus}\{\hat{0},\hat{1}\}$. Then, $\deg_\calG(\hat{0})=\deg_\calG(\hat{1})=2$. 
Let $v_1, v_2$ (resp. $w_1, w_2$) be the vertices which are incident to $\hat{0}$ (resp. $\hat{1}$), and assume that there is a maximal chain 
$v_1=u_1\prec u_2\prec\cdots\prec u_\ell=w_1$ with $u_i\in V(\calG){\setminus}\{\hat{0},\hat{1}\}$ for any $i=1,\cdots, \ell$. 
     \begin{itemize}
     \item We suppose that one of $v,w$ is contained in $\{u_1,\cdots,u_\ell\}$. We assume that $v\in\{u_1,\cdots,u_\ell\}$, and let $v=u_i$. 
     Then, $\deg_\calG u_i=3$ and $\deg_\calG u_j=2$ for any $j\neq i$. 
     Let $U(v)$ (resp. $D(v)$) be the set of edges incident to $v$ and connecting $v$ to $v^\prime$ with $v\prec v^\prime$ (resp. $v^\prime\prec v$). 
     If $|U(v)|=2$ and $|D(v)|=1$, then $|U(w)|=1$ and $|D(w)|=2$ holds. 
     Thus, there should be a maximal chain $v_2=u_1^\prime\prec\cdots\prec u_\ell^\prime=w_2$ with $w=u_i^\prime$ for some $1\le i\le\ell$, and hence we see that $\calG$ looks like (II). 
     If $|U(v)|=1$ and $|D(v)|=2$, we also have the same type by the same argument. 
     \item If $v,w\not\in\{u_1,\cdots,u_\ell\}$ then $\deg_\calG u_i=2$ for any $i=1,\cdots, \ell$. Thus, there should be a maximal chain $v_2=u_1^\prime\prec\cdots\prec u_{\ell}^\prime=w_2$ and $\deg_\calG u_s^\prime=\deg_\calG u_t^\prime=3$ for some $1\le s< t\le\ell$. Thus, we see that $\calG$ looks like (III). 
     \item If $v, w\in\{u_1,\cdots,u_\ell\}$, then we have $\deg_\calG u_s=\deg_\calG u_t=3$ for some $1\le s< t\le\ell$ and $\deg_\calG u_i=2$ for any $i\neq s,t$. We may assume that $v=u_s, w=u_t$, and hence $v\prec w$. 
     \begin{itemize}
     \item If $|U(v)|=2$ and $|D(v)|=1$, then $|U(w)|=1$ and $|D(w)|=2$ holds. Thus, in this case, we also have the graph $\calG$ looks like (III). 
     \item If $|U(v)|=1$ and $|D(v)|=2$, then $|U(w)|=2$ and $|D(w)|=1$. In this case, any maximal chain contains edges corresponding to $U(v),D(w)$ (and edges connecting them). 
     Thus, the associated Hibi ring is the polynomial extension of a certain Hibi ring. (More precisely, we see that it is the polynomial extension of a Hibi ring associated with a poset of the type (IV).) 
     \end{itemize}
     \end{itemize}
\end{itemize}
     
\noindent{\bf The case (ii)}: Let $v$ be the vertex of $\calG$ with degree $4$. Then, it is easy to see that $v\not\in\{\hat{0},\hat{1}\}$. 
Moreover, we also see that $|U(v)|=|D(v)|=2$, and hence $\calG$ looks like (IV). 
\end{proof}

We then show our main theorem. 

\begin{theorem}
\label{main_thm_HibiZ2}
Let $R=\kk[P]$ be a Gorenstein Hibi ring with $\Cl(R)\cong\ZZ^2$, in which case $\calH(\widehat{P})$ is one of {\rm(I)--(V)} in Figure~\ref{figure_poset_Z2}. 
For each type of these posets, we define the subset $\calL$ of the characters as follows. 
\begin{itemize}
\item [$\bullet$] When $\calH(\widehat{P})$ is the type {\rm(I)}, we define the number of edges as in Figure~\ref{poset_typeI}, and let 
\[
\calL=\{\chi=(c_1,c_2)\in\rmX(G)\mid 0\le c_1\le m+n+1,\, 0\le c_2\le n \}. 
\]
\item [$\bullet$] When $\calH(\widehat{P})$ is the type {\rm(II)}, we define the number of edges as in Figure~\ref{poset_typeII}, and let 
\[
\calL=\{\chi=(c_1,c_2)\in\rmX(G)\mid 0\le c_1\le\ell+m,\, 0\le c_2\le m+n\}. 
\]
\item [$\bullet$] When $\calH(\widehat{P})$ is the type {\rm(III)}, we define the number of edges as in Figure~\ref{poset_typeIII}, and let 
\[
\calL=\{\chi=(c_1,c_2)\in\rmX(G)\mid 0\le c_1\le \ell+m+n+1,\, 0\le c_2\le m-1\}. 
\]
\item [$\bullet$] When $\calH(\widehat{P})$ is the type {\rm(IV)}, we define the number of edges as in Figure~\ref{poset_typeIV}, and let 
\[
\calL=\{\chi=(c_1,c_2)\in\rmX(G)\mid 0\le c_1\le m,\, 0\le c_2\le n\}. 
\]
\item [$\bullet$] When $\calH(\widehat{P})$ is the type {\rm(V)}, we define the number of edges as in Figure~\ref{poset_typeV}, and let 
\[
\calL=\{\chi=(c_1,c_2)\in\rmX(G)\mid 0\le c_1\le n+1,\, 0\le c_2\le n+1 \}. 
\]
\end{itemize}
Then, we have that $\End_R(M_\calL)$ is a splitting NCCR of $R$, where $M_\calL\coloneqq\bigoplus_{\chi\in\calL}M_\chi$ and $M_\chi=(S\otimes_\kk V_\chi)^G$ is the module of covariants associated to $\chi$. 
\end{theorem}

We note that applying the mutations (see \cite[Section~6]{IW1}, \cite[Section~4]{HN}) repeatedly to the module $M_\calL$ in Theorem~\ref{main_thm_HibiZ2}, 
we can obtain a lot of modules giving NCCRs of $R$. 
We remark that the mutated one is not necessarily splitting. 
On the other hand, for a module $M$ giving a splitting NCCR, $(M\otimes_RI)^{**}$ also gives a splitting NCCR where $I$ is a divisorial ideal and $(-)^*=\Hom_R(-,R)$. 
Therefore, starting from a given one, we can obtain infinitely many modules giving splitting NCCRs. 

\medskip

We will prove Theorem~\ref{main_thm_HibiZ2} by a case-by-case check along the classification of posets in Lemma~\ref{poset_Z2}.  
Since the argument is the same for each type, we will give the precise proof for the type (I) in Subsection~\ref{subsec_proof_typeI} and we only mention an outline for other cases in Subsection~\ref{subsec_proof_othertype}. 
Also, in the proof of this theorem, we use the classification of rank one MCM modules for each Hibi ring, thus we will give this classification in Section~\ref{sec_MCM_HibiZ2}. 

We note that if $\calH(\widehat{P})$ is the type (IV), then the Hibi ring $R$ is quasi-symmetric, thus the assertion also follows from \cite[Theorem~1.19]{SpVdB}. 
On the other hand, if $\calH(\widehat{P})$ is the type (V), then the Hibi ring $R$ is isomorphic to the Segre product of three polynomial rings with $n+2$ variables (see \cite[Example~2.6]{HN}), in which case Theorem~\ref{main_thm_HibiZ2} is a special case of \cite[Theorem~3.6]{HN}. 

\subsection{Methods for determining the global dimension}
\label{subsec_keylem}

In order to construct NCCRs, some Lemmas in \cite[Section~10]{SpVdB} are main ingredients, thus we introduce several notions following that paper. 

\begin{observation}
\label{obs_GSmod}
Although we are interested in Hibi rings, the following argument holds in more general context, thus let $R$ be a toric ring with $\Cl(R)\cong\ZZ^r$, and use the same notation as in Subsection~\ref{sec_preliminary_toric}.  
We consider the category $\calA=\mc(G,S)$ of finitely generated $(G,S)$-module (i.e., $G$-equivariant $S$-modules). 
Then, any projective generator of $\calA$ can be given by $P_\chi\coloneqq V_\chi\otimes_\kk S$ with $\chi\in\rmX(G)$. 
For a finite subset $\calL$ of $\rmX(G)$, we set 
$$
P_\calL\coloneqq\bigoplus_{\chi\in\calL}P_\chi \text{\quad and \quad} \Lambda_\calL\coloneqq\End_\calA(P_\calL). 
$$
For $\chi\in\rmX(G)$, we set $P_{\calL,\chi}\coloneqq\Hom_\calA(P_\calL,P_\chi)$, and this is a right projective $\Lambda_\calL$-module if $\chi\in\calL$. 
Using an equivalence $(-)^G:\refl(G,S)\xrightarrow{\simeq}\refl(R)$ (see e.g., \cite[Lemma~3.3]{SpVdB}) between the category of reflexive $(G,S)$-modules and that of reflexive $R$-modules, we see that 
$$
\Hom_\calA(P_{\chi_i},P_{\chi_j})\cong(\Hom_k(V_{\chi_i},V_{\chi_j})\otimes_\kk S)^G\cong\Hom_R(P_{\chi_i}^G,P_{\chi_j}^G)=\Hom_R(M_{\chi_i},M_{\chi_j}). 
$$
Thus, to determine the global dimension of $\End_R(P_\calL^G)=\End_R(M_\calL)$, we may only consider that of $\Lambda_\calL$. 
\end{observation}

In particular, the following lemma is useful. 

\begin{lemma}[{see \cite[Lemma~10.1]{SpVdB}}]
\label{key_lem1}
We have that $\gldim\Lambda_\calL<\infty$ if and only if $\pdim_{\Lambda_\calL}P_{\calL,\chi}<\infty$ for all $\chi\in\rmX(G)$. 
\end{lemma}

We then consider a one-parameter subgroup $\lambda^\bfb:\kk^\times\rightarrow G=(\kk^\times)^r$ of $G$ which is a group homomorphism defined as $\lambda^\bfb(g)=(g^{b_1},\cdots,g^{b_r})$ for $g\in G$ and $\bfb=(b_1,\cdots,b_r)\in\ZZ^r$. 
In particular, the group $\rmY(G)$ of one-parameter subgroups of $G$ is isomorphic to $\ZZ^r$ by identifying a one-parameter subgroup $\lambda^\bfb$ with its weight $\bfb\in\ZZ^r$. 
The following notion is important for studying the global dimension of $\Lambda_\calL$. 

\begin{definition}
\label{def_separated}
For a finite subset $\calL\subset\rmX(G)$, we say that $\chi\in\rmX(G)$ is \emph{separated} from $\calL$ by $\lambda\in\rmY(G)_\RR\coloneqq\rmY(G)\otimes_\ZZ\RR$ if $\langle\lambda,\chi\rangle<\langle\lambda,\nu\rangle$ for any $\nu\in\calL$. 
\end{definition}

For $\lambda\in\rmY(G)$, let $K_\lambda$ be the subspace of $W=\bigoplus_{i=1}^nV_{\beta_i}$ spanned by representations corresponding to $\beta_{i_j}$ with $\langle\lambda,\beta_{i_j}\rangle>0$, 
and let $d_\lambda\coloneqq\dim_\kk K_\lambda$. 
We then consider the Koszul resolution:  
$$
0\longrightarrow \wedge^{d_\lambda}K_\lambda\otimes_\kk S \longrightarrow\wedge^{d_\lambda-1}K_\lambda\otimes_\kk S \longrightarrow \cdots \longrightarrow S \longrightarrow S(W/K_\lambda)\longrightarrow 0.
$$ 
Applying $(V_\chi\otimes_\kk-)$ to this sequence, we have the following exact sequence: 
\begin{align}
\label{chi_sequence}
C_{\lambda,\chi}: 0\longrightarrow (V_\chi\otimes\wedge^{d_\lambda}K_\lambda)\otimes S \overset{\delta_{d_\lambda}}{\longrightarrow}(V_\chi&\otimes\wedge^{d_\lambda-1}K_\lambda)\otimes S \overset{\delta_{d_{\lambda-1}}}{\longrightarrow} \cdots \\ 
&\cdots\overset{\delta_1}{\longrightarrow} V_\chi\otimes S \longrightarrow V_\chi\otimes S(W/K_\lambda)\longrightarrow 0. \nonumber
\end{align}
Here, we note that for $p=1,\cdots, d_\lambda$ the $(G,S)$-module $(V_\chi\otimes_\kk\wedge^pK_\lambda)\otimes_\kk S$ is decomposed as the direct sum of $(G,S)$-modules $P_\mu$ 
with the form $\mu=\chi+\beta_{i_1}+\cdots+\beta_{i_p}$ where $\{i_1, \cdots, i_p\}\subset \{1, \cdots, n\}$, 
$i_j\neq i_{j^\prime}$ if $j\neq j^\prime$, and $\langle\lambda,\beta_{i_j}\rangle>0$. 
Then we have the following lemma. 

\begin{lemma}[{see \cite[Lemma~10.2]{SpVdB}}]
\label{key_lem2}
We assume that $\chi\in\rmX(G)$ is separated from $\calL$ by $\lambda\in\rmY(G)_\RR$. 
Then, we see that the complex $C_{\calL,\lambda,\chi}=\Hom_\calA(P_\calL,C_{\lambda,\chi})$ is acyclic. 
In addition, the $0$-th term of $C_{\calL,\lambda,\chi}$ is $P_{\calL,\chi}$ and for $p=1,\cdots, d_\lambda$ 
the $-p$-th term is the direct sum of $P_{\calL,\mu}$ with the form 
$$
\mu=\chi+\beta_{i_1}+\cdots+\beta_{i_p}
$$
where $\{i_1, \cdots, i_p\}\subset \{1, \cdots, n\}$, $i_j\neq i_{j^\prime}$ if $j\neq j^\prime$, and $\langle\lambda,\beta_{i_j}\rangle>0$. 
\end{lemma}

\medskip

\subsection{Splitting NCCRs for the type (I)}
\label{subsec_proof_typeI}

In the rest of this section, using Lemma~\ref{key_lem1} and \ref{key_lem2}, we will show the finiteness of global dimension of $\End_R(M_\calL)$ 
for each Gorenstein Hibi ring $R$ with $\Cl(R)\cong\ZZ^2$ given in Theorem~\ref{main_thm_HibiZ2}. 
We will give the proof for the type (I) in this subsection and we only mention an outline for other cases in the next subsection. 
Thus, we first consider the Gorenstein Hibi ring $R=\kk[P]$ associated with the poset $\widehat{P}$ shown in Figure~\ref{poset_typeI}. 
In particular, we fix grayed edges as a spanning tree of $\widehat{P}$.  

\begin{figure}[H]
\begin{center}
\newcommand{\edgewidth}{0.06cm} 
\newcommand{\nodewidth}{0.06cm} 
\newcommand{\noderad}{0.2} 
\newcommand{\pmwidth}{0.65cm} 
\newcommand{\pmcolor}{lightgray} 

\begin{tikzpicture}
\node at (0,0)
{\scalebox{0.3}{
\begin{tikzpicture}
\coordinate (Min) at (4,0); \coordinate (Max) at (4,15);
\coordinate (N11) at (0,2); \coordinate (N12) at (0,7); \coordinate (N13) at (0,8); \coordinate (N14) at (0,13); 
\coordinate (N21) at (7,2); \coordinate (N22) at (7,7); 
\coordinate (N23a) at (6,8);  \coordinate (N24a) at (6,13); \coordinate (N23b) at (8,8); \coordinate (N24b) at (8,13); 

\draw[line width=\pmwidth, \pmcolor]  (N11)--(N14)--(Max) ; 
\draw[line width=\pmwidth, \pmcolor]  (Min)--(N21)--(N22)--(N23a)--(N24a)--(Max) ; 
\draw[line width=\pmwidth, \pmcolor]  (N22)--(N23b)--(N24b) ; 

\draw[line width=\edgewidth]  (Min)--(N11); \draw[line width=\edgewidth]  (N11)--(0,2.8); 
\draw[line width=\edgewidth, loosely dotted]  (0,3.4)--(0,5.6); \draw[line width=\edgewidth]  (0,6.2)--(N12); 
\draw[line width=\edgewidth]  (N12)--(N13); \draw[line width=\edgewidth]  (N13)--(0,8.8); 
\draw[line width=\edgewidth, loosely dotted]  (0,9.4)--(0,11.6); \draw[line width=\edgewidth]  (0,12.2)--(N14); \draw[line width=\edgewidth]  (N14)--(Max); 

\draw[line width=\edgewidth]  (Min)--(N21); \draw[line width=\edgewidth]  (N21)--(7,2.8); 
\draw[line width=\edgewidth, loosely dotted]  (7,3.4)--(7,5.6); \draw[line width=\edgewidth]  (7,6.2)--(N22); 

\draw[line width=\edgewidth]  (N22)--(N23a); \draw[line width=\edgewidth]  (N22)--(N23b); 
\draw[line width=\edgewidth]  (N23a)--(6,8.8); \draw[line width=\edgewidth]  (N23b)--(8,8.8); 
\draw[line width=\edgewidth, loosely dotted]  (6,9.4)--(6,11.6); \draw[line width=\edgewidth, loosely dotted]  (8,9.4)--(8,11.6); 
\draw[line width=\edgewidth]  (6,12.2)--(N24a); \draw[line width=\edgewidth]  (N24a)--(Max); 
\draw[line width=\edgewidth]  (8,12.2)--(N24b); \draw[line width=\edgewidth]  (N24b)--(Max); 

\draw [line width=\nodewidth, fill=gray] (Min) circle [radius=\noderad] ; \draw [line width=\nodewidth, fill=gray] (Max) circle [radius=\noderad] ; 
\draw [line width=\nodewidth, fill=white] (N11) circle [radius=\noderad] ; \draw [line width=\nodewidth, fill=white] (N12) circle [radius=\noderad] ; 
\draw [line width=\nodewidth, fill=white] (N13) circle [radius=\noderad] ; \draw [line width=\nodewidth, fill=white] (N14) circle [radius=\noderad] ; 
\draw [line width=\nodewidth, fill=white] (N21) circle [radius=\noderad] ; \draw [line width=\nodewidth, fill=white] (N22) circle [radius=\noderad] ; 
\draw [line width=\nodewidth, fill=white] (N23a) circle [radius=\noderad] ; \draw [line width=\nodewidth, fill=white] (N24a) circle [radius=\noderad] ; 
\draw [line width=\nodewidth, fill=white] (N23b) circle [radius=\noderad] ; \draw [line width=\nodewidth, fill=white] (N24b) circle [radius=\noderad] ; 

\draw [line width=0.03cm, decorate, decoration={brace,amplitude=10pt}](-0.8,2) -- (-0.8,7) node[black,midway,xshift=-0.8cm,yshift=0cm] {\Huge $m$} ; 
\draw [line width=0.03cm, decorate, decoration={brace,amplitude=10pt}](-0.8,7) -- (-0.8,13) node[black,midway,xshift=-0.8cm,yshift=0cm] {\Huge $n$} ; 
\draw [line width=0.03cm, decorate, decoration={brace, mirror, amplitude=10pt}](7.8,2) -- (7.8,7) node[black,midway,xshift=0.8cm,yshift=0cm] {\Huge $m$} ; 
\draw [line width=0.03cm, decorate, decoration={brace, mirror, amplitude=10pt}](8.8,7) -- (8.8,13) node[black,midway,xshift=0.8cm,yshift=0cm] {\Huge $n$} ; 
\node at (2,0.5) {\Huge$a$} ; \node at (6.2,14.5) {\Huge$b$} ; 

\end{tikzpicture}
} }; 

\node at (6.5,0) {
\begin{tabular}{c||c}
weight & the number of weights \\ \hline
$(1,0)$& $m+n+2$\\
$(0,1)$& $n+1$ \\
$(\mathchar`-1,0)$& $m+1$ \\
$(\mathchar`-1,\mathchar`-1)$& $n+1$
\end{tabular}
};
\end{tikzpicture}
\end{center}
\caption{The type (I) poset in Figure~\ref{figure_poset_Z2}, and the list of weights. Here, $m, n$ are the number of edges, where $m\ge0, n\ge 1$.}
\label{poset_typeI}
\end{figure}

Then, the prime divisors corresponding to edges not contained in this spanning tree, which is denoted by $a,b$, generate $\Cl(R)$. 
We denote such divisors by $\calD_a,\calD_b$ respectively, and denote the corresponding weight by $\beta_a=(1,0)$ and $\beta_b=(0,1)$. 
Using relations (\ref{relation_div_hibi}), we have the list of weights corresponding to prime divisors as shown in Figure~\ref{poset_typeI}. 
Then, we can prove Theorem~\ref{main_thm_HibiZ2}. 

\begin{proof}[The proof of Theorem~\ref{main_thm_HibiZ2} for the type (I)]
First, we show that $\gldim \End_R(\bigoplus_{\chi\in\calL}M_\chi)<\infty$. 
By Observation~\ref{obs_GSmod}, we may only consider the global dimension of $\Lambda_\calL$. 
To show this, we may show that $\pdim_{\Lambda_\calL}P_{\calL,\chi}<\infty$ for all $\chi\in\rmX(G)$ by Lemma~\ref{key_lem1}. 
We recall that if $\chi\in\calL$, then $P_{\calL,\chi}$ is a projective $\Lambda_\calL$-module. Let 
$$
\calM_j=\{\chi=(c_1,c_2)\in\rmX(G) \mid 0\le c_1\le m+n+1,\, c_2=-j\} 
$$
and take $\lambda=(0,1)\in\rmY(G)$. For $j\in\ZZ_{>0}$, any character $\chi\in\calM_j$ is separated from $\calL$ by $\lambda$, 
and the weight $\beta_i$ of characters satisfying $\langle\lambda,\beta_i\rangle>0$ are only $(0,1)$ with multiplicity $n+1$. 

When $j=1$, we see that $\pdim_{\Lambda_\calL}P_{\calL,\chi}<\infty$ for any $\chi\in\calM_1$ by using Lemma~\ref{key_lem2}. 
In fact, for $\chi\in\calM_1$, we have the acyclic complex $C_{\calL,\lambda,\chi}$ as in Lemma~\ref{key_lem2}, 
and we see that each component of the $-p$-th term $(p=1, \cdots, d_\lambda)$ is projective. Therefore, $C_{\calL,\lambda,\chi}$ is just a projective resolution of $P_{\calL,\chi}$. 

We then assume that $\pdim_{\Lambda_\calL}P_{\calL,\chi}<\infty$ for any $\chi\in\calM_k$ with $k=1, \cdots, j-1$. 
For any $\chi\in\calM_j$, we again have the acyclic complex $C_{\calL,\lambda,\chi}$, and see that each component of the $-p$-th term $(p=1, \cdots, d_\lambda)$ has the finite projective dimension by the assumption. 
Therefore, we have that $\pdim_{\Lambda_\calL}P_{\calL,\chi}<\infty$ for any $\chi\in\calM_j$ with $j\in\ZZ_{> 0}$. 

Then, we let 
$$
\calN_j=\{\chi=(c_1,c_2)\in\rmX(G) \mid c_1= -j,\, c_2\le n\}. 
$$
For $j\in\ZZ_{> 0}$, any character $\chi\in\calN_j$ is separated from $\calL$ by $\lambda=(1,0)\in\rmY(G)$,  
and the weight $\beta_i$ of characters satisfying $\langle\lambda,\beta_i\rangle>0$ are only $(1,0)$ with multiplicity $m+n+2$. 
By the same inductive arguments used in the case of $\calM_j$, we see that $\pdim_{\Lambda_\calL}P_{\calL,\chi}<\infty$ for any $\chi\in\calN_j$ with $j\in\ZZ_{> 0}$. 

Since the conic region $\calC(P)$ (see Theorem~\ref{conic_Hibi} or the proof of Proposition~\ref{prop_MCM_typeI} for the precise description of $\calC(P)$)  is contained in $\calL\cup\bigcup_{j\in\ZZ_{> 0}}\calM_j\cup\bigcup_{j\in\ZZ_{> 0}}\calN_j$, 
we especially have that $\pdim_{\Lambda_\calL}P_{\calL,\chi}<\infty$ for any $\chi\in\calC(P)$. 
This is enough to show $\pdim_{\Lambda_\calL}P_{\calL,\chi}<\infty$ for any $\chi\in\rmX(G)$ by the argument in \cite[Subsection~10.3]{SpVdB}. 
Thus, we have that $\gldim\Lambda_\calL<\infty$ and hence $\gldim\End_R(\bigoplus_{\chi\in\calL}M_\chi)<\infty$. 

Since $\Hom_R(M_\chi,M_{\chi^\prime})\cong M_{\chi^\prime-\chi}$, we see that $\End_R(\bigoplus_{\chi\in\calL}M_\chi)$ is MCM by the complete list of rank one MCM modules which will be given in Proposition~\ref{prop_MCM_typeI}. 
Thus, we have that $\End_R(\bigoplus_{\chi\in\calL}M_\chi)$ is an NCCR of $R$. 
\end{proof}

\subsection{Splitting NCCRs for other types}
\label{subsec_proof_othertype}

We then prove Theorem~\ref{main_thm_HibiZ2} for Gorenstein Hibi rings associated with posets of the type (II)--(V). 
The proof can be done by combining the same arguments as in Subsection~\ref{subsec_proof_typeI} and Proposition~\ref{prop_MCM_typeII}--\ref{prop_MCM_typeV}. Thus, we only mention an outline here. 

Let $R=\kk[P]$ be the Gorenstein Hibi ring associated with one of the posets $\widehat{P}$ shown in Figure~\ref{poset_typeII}--\ref{poset_typeV} below. 
For these figures, we fix grayed edges as a spanning tree of $\widehat{P}$. 
In this situation, the prime divisors $\calD_a,\calD_b$ corresponding to edges $a,b$, which are not contained in a fixed spanning tree, generate $\Cl(R)$. 
We denote the weights corresponding to $\calD_a,\calD_b$ by $\beta_a=(1,0), \beta_b=(0,1)$ respectively. 
Then, by using the relations (\ref{relation_div_hibi}), we have the list of weights corresponding to prime divisors as given in each figure below. 
We need only these data for applying the arguments in Subsection~\ref{subsec_proof_typeI}, thus we can prove our assertions. 

\begin{figure}[H]
\begin{center}
\newcommand{\edgewidth}{0.06cm} 
\newcommand{\nodewidth}{0.06cm} 
\newcommand{\noderad}{0.2} 
\newcommand{\pmwidth}{0.65cm} 
\newcommand{\pmcolor}{lightgray} 

\begin{tikzpicture}
\node at (0,0)
{\scalebox{0.3}{
\begin{tikzpicture}

\coordinate (Min) at (4,0); \coordinate (Max) at (4,15);
\coordinate (N11) at (0,2); \coordinate (N12) at (0,5); \coordinate (N13) at (0,6); \coordinate (N14) at (0,9); \coordinate (N15) at (0,10); \coordinate (N16) at (0,13); 
\coordinate (N21) at (8,2); \coordinate (N22) at (8,5); \coordinate (N23) at (8,6); \coordinate (N24) at (8,9); \coordinate (N25) at (8,10); \coordinate (N26) at (8,13); 
\coordinate (N31) at (2,8.75); \coordinate (N32) at (6,6.25);

\draw[line width=\pmwidth, \pmcolor]  (N11)--(N16)--(Max) ; \draw[line width=\pmwidth, \pmcolor]  (Min)--(N21)--(N26) ; 
\draw[line width=\pmwidth, \pmcolor]  (N15)--(N22) ; 

\coordinate (Min) at (4,0); \coordinate (Max) at (4,15);
\coordinate (N11) at (0,2); \coordinate (N12) at (0,5); \coordinate (N13) at (0,6); \coordinate (N14) at (0,9); \coordinate (N15) at (0,10); \coordinate (N16) at (0,13); 
\coordinate (N21) at (8,2); \coordinate (N22) at (8,5); \coordinate (N23) at (8,6); \coordinate (N24) at (8,9); \coordinate (N25) at (8,10); \coordinate (N26) at (8,13); 

\coordinate (N31) at (2,8.75); \coordinate (N32) at (6,6.25);

\draw[line width=\edgewidth]  (Min)--(N11); \draw[line width=\edgewidth]  (N11)--(0,2.8); \draw[line width=\edgewidth, loosely dotted]  (0,3)--(0,4); 
\draw[line width=\edgewidth]  (0,4.2)--(N12); \draw[line width=\edgewidth]  (N12)--(N13); \draw[line width=\edgewidth]  (N13)--(0,6.8); 
\draw[line width=\edgewidth, loosely dotted]  (0,7)--(0,8); \draw[line width=\edgewidth]  (0,8.2)--(N14); \draw[line width=\edgewidth]  (N14)--(N15); 
\draw[line width=\edgewidth]  (N15)--(0,10.8); \draw[line width=\edgewidth, loosely dotted]  (0,11)--(0,12); 
\draw[line width=\edgewidth]  (0,12.2)--(N16); \draw[line width=\edgewidth]  (N16)--(Max); 

\draw[line width=\edgewidth]  (Min)--(N21); \draw[line width=\edgewidth]  (N21)--(8,2.8); \draw[line width=\edgewidth, loosely dotted]  (8,3)--(8,4); 
\draw[line width=\edgewidth]  (8,4.2)--(N22); \draw[line width=\edgewidth]  (N22)--(N23); \draw[line width=\edgewidth]  (N23)--(8,6.8); 
\draw[line width=\edgewidth, loosely dotted]  (8,7)--(8,8); \draw[line width=\edgewidth]  (8,8.2)--(N24); \draw[line width=\edgewidth]  (N24)--(N25); 
\draw[line width=\edgewidth]  (N25)--(8,10.8); \draw[line width=\edgewidth, loosely dotted]  (8,11)--(8,12); 
\draw[line width=\edgewidth]  (8,12.2)--(N26); \draw[line width=\edgewidth]  (N26)--(Max); 

\draw[line width=\edgewidth]  (N15)--(N31); \draw[line width=\edgewidth]  (N31)--(2.8,8.25); 
\draw[line width=\edgewidth, loosely dotted]  (3.2,8)--(4.8,7); 
\draw[line width=\edgewidth]  (5.2,6.75)--(N32); \draw[line width=\edgewidth]  (N32)--(N22); 

\draw [line width=\nodewidth, fill=gray] (Min) circle [radius=\noderad] ; \draw [line width=\nodewidth, fill=gray] (Max) circle [radius=\noderad] ; 
\draw [line width=\nodewidth, fill=white] (N11) circle [radius=\noderad] ; \draw [line width=\nodewidth, fill=white] (N12) circle [radius=\noderad] ; \draw [line width=\nodewidth, fill=white] (N13) circle [radius=\noderad] ; 
\draw [line width=\nodewidth, fill=white] (N14) circle [radius=\noderad] ; \draw [line width=\nodewidth, fill=white] (N15) circle [radius=\noderad] ; \draw [line width=\nodewidth, fill=white] (N16) circle [radius=\noderad] ; 
\draw [line width=\nodewidth, fill=white] (N21) circle [radius=\noderad] ; \draw [line width=\nodewidth, fill=white] (N22) circle [radius=\noderad] ; \draw [line width=\nodewidth, fill=white] (N23) circle [radius=\noderad] ; 
\draw [line width=\nodewidth, fill=white] (N24) circle [radius=\noderad] ; \draw [line width=\nodewidth, fill=white] (N25) circle [radius=\noderad] ; \draw [line width=\nodewidth, fill=white] (N26) circle [radius=\noderad] ; 
\draw [line width=\nodewidth, fill=white] (N31) circle [radius=\noderad] ; \draw [line width=\nodewidth, fill=white] (N32) circle [radius=\noderad] ; 

\draw [line width=0.03cm, decorate, decoration={brace,amplitude=10pt}](-0.8,2) -- (-0.8,5) node[black,midway,xshift=-0.8cm,yshift=0cm] {\Huge $\ell$} ; 
\draw [line width=0.03cm, decorate, decoration={brace,amplitude=10pt}](-0.8,5) -- (-0.8,10) node[black,midway,xshift=-0.8cm,yshift=0cm] {\Huge $m$} ; 
\draw [line width=0.03cm, decorate, decoration={brace,amplitude=10pt}](-0.8,10) -- (-0.8,13) node[black,midway,xshift=-0.8cm,yshift=0cm] {\Huge $n$} ; 
\draw [line width=0.03cm, decorate, decoration={brace, mirror, amplitude=10pt}](8.8,2) -- (8.8,5) node[black,midway,xshift=0.8cm,yshift=0cm] {\Huge $\ell$} ; 
\draw [line width=0.03cm, decorate, decoration={brace, mirror, amplitude=10pt}](8.8,5) -- (8.8,10) node[black,midway,xshift=0.8cm,yshift=0cm] {\Huge $m$} ; 
\draw [line width=0.03cm, decorate, decoration={brace, mirror, amplitude=10pt}](8.8,10) -- (8.8,13) node[black,midway,xshift=0.8cm,yshift=0cm] {\Huge $n$} ; 
\draw [line width=0.03cm, decorate, decoration={brace,amplitude=10pt}](0.3,10.2) -- (7.7,5.575) node[black,midway,xshift=0.25cm,yshift=0.7cm] {\Huge $m$} ; 
\node at (2,0.5) {\Huge$a$} ; \node at (6,14.5) {\Huge$b$} ; 
\end{tikzpicture}
} }; 

\node at (6.5,0) {
\begin{tabular}{c||c}
weight & the number of weights \\ \hline
$(1,0)$& $\ell+m+1$\\
$(0,1)$& $m+n+1$ \\
$(\mathchar`-1,0)$& $\ell+1$ \\
$(0,\mathchar`-1)$& $n+1$ \\
$(\mathchar`-1,\mathchar`-1)$& $m$ 
\end{tabular}
};
\end{tikzpicture}
\end{center}
\caption{The type (II) poset in Figure~\ref{figure_poset_Z2}, and the list of weights. Here, $\ell, m, n$ are the number of edges, where $\ell\ge 0, m\ge1, n\ge 0$.}
\label{poset_typeII}
\end{figure}

\begin{figure}[H]
\begin{center}
\newcommand{\edgewidth}{0.06cm} 
\newcommand{\nodewidth}{0.06cm} 
\newcommand{\noderad}{0.2} 
\newcommand{\pmwidth}{0.65cm} 
\newcommand{\pmcolor}{lightgray} 

\begin{tikzpicture}
\node at (0,0)
{\scalebox{0.3}{
\begin{tikzpicture}
\coordinate (Min) at (4,0); \coordinate (Max) at (4,15);
\coordinate (N11) at (0,2); \coordinate (N12) at (0,5); \coordinate (N13) at (0,6); 
\coordinate (N14) at (0,9); \coordinate (N15) at (0,10); \coordinate (N16) at (0,13); 

\coordinate (N21) at (7,2); \coordinate (N22) at (7,5); 
\coordinate (N23a) at (6,6);  \coordinate (N24a) at (6,9); \coordinate (N23b) at (8,6); \coordinate (N24b) at (8,9); 
\coordinate (N25) at (7,10); \coordinate (N26) at (7,13); 

\draw[line width=\pmwidth, \pmcolor]  (N11)--(N16)--(Max) ; 
\draw[line width=\pmwidth, \pmcolor]  (Min)--(N21)--(N22)--(N23a)--(N24a)--(N25)--(N26)--(Max) ; 
\draw[line width=\pmwidth, \pmcolor]  (N22)--(N23b)--(N24b) ; 

\draw[line width=\edgewidth]  (Min)--(N11); \draw[line width=\edgewidth]  (N11)--(0,2.8); 
\draw[line width=\edgewidth, loosely dotted]  (0,3)--(0,4); 
\draw[line width=\edgewidth]  (0,4.2)--(N12); \draw[line width=\edgewidth]  (N12)--(N13); \draw[line width=\edgewidth]  (N13)--(0,6.8); 
\draw[line width=\edgewidth, loosely dotted]  (0,7)--(0,8); \draw[line width=\edgewidth]  (0,8.2)--(N14); 
\draw[line width=\edgewidth]  (N14)--(N15); \draw[line width=\edgewidth]  (N15)--(0,10.8); 
\draw[line width=\edgewidth, loosely dotted]  (0,11)--(0,12); \draw[line width=\edgewidth]  (0,12.2)--(N16); \draw[line width=\edgewidth]  (N16)--(Max); 

\draw[line width=\edgewidth]  (Min)--(N21); \draw[line width=\edgewidth]  (N21)--(7,2.8); 
\draw[line width=\edgewidth, loosely dotted]  (7,3)--(7,4); 
\draw[line width=\edgewidth]  (7,4.2)--(N22); \draw[line width=\edgewidth]  (N22)--(N23a); \draw[line width=\edgewidth]  (N22)--(N23b); 
\draw[line width=\edgewidth]  (N23a)--(6,6.8); \draw[line width=\edgewidth]  (N23b)--(8,6.8); 
\draw[line width=\edgewidth, loosely dotted]  (6,7)--(6,8); \draw[line width=\edgewidth, loosely dotted]  (8,7)--(8,8); 
\draw[line width=\edgewidth]  (6,8.2)--(N24a); \draw[line width=\edgewidth]  (8,8.2)--(N24b); 
\draw[line width=\edgewidth]  (N24a)--(N25); \draw[line width=\edgewidth]  (N24b)--(N25); \draw[line width=\edgewidth]  (N25)--(7,10.8); 
\draw[line width=\edgewidth, loosely dotted]  (7,11)--(7,12); \draw[line width=\edgewidth]  (7,12.2)--(N26); \draw[line width=\edgewidth]  (N26)--(Max); 

\draw [line width=\nodewidth, fill=gray] (Min) circle [radius=\noderad] ; \draw [line width=\nodewidth, fill=gray] (Max) circle [radius=\noderad] ; 
\draw [line width=\nodewidth, fill=white] (N11) circle [radius=\noderad] ; \draw [line width=\nodewidth, fill=white] (N12) circle [radius=\noderad] ; 
\draw [line width=\nodewidth, fill=white] (N13) circle [radius=\noderad] ; \draw [line width=\nodewidth, fill=white] (N14) circle [radius=\noderad] ; 
\draw [line width=\nodewidth, fill=white] (N15) circle [radius=\noderad] ; \draw [line width=\nodewidth, fill=white] (N16) circle [radius=\noderad] ; 
\draw [line width=\nodewidth, fill=white] (N21) circle [radius=\noderad] ; \draw [line width=\nodewidth, fill=white] (N22) circle [radius=\noderad] ; 
\draw [line width=\nodewidth, fill=white] (N23a) circle [radius=\noderad] ; \draw [line width=\nodewidth, fill=white] (N24a) circle [radius=\noderad] ; 
\draw [line width=\nodewidth, fill=white] (N23b) circle [radius=\noderad] ; \draw [line width=\nodewidth, fill=white] (N24b) circle [radius=\noderad] ; 
\draw [line width=\nodewidth, fill=white] (N25) circle [radius=\noderad] ; \draw [line width=\nodewidth, fill=white] (N26) circle [radius=\noderad] ; 

\draw [line width=0.03cm, decorate, decoration={brace,amplitude=10pt}](-0.8,2) -- (-0.8,13) node[black,midway,xshift=-2.3cm,yshift=0cm] {\Huge $\ell+m+n$} ; 
\draw [line width=0.03cm, decorate, decoration={brace, mirror, amplitude=10pt}](7.8,2) -- (7.8,5) node[black,midway,xshift=0.8cm,yshift=0cm] {\Huge $\ell$} ; 
\draw [line width=0.03cm, decorate, decoration={brace, mirror, amplitude=10pt}](8.8,5) -- (8.8,10) node[black,midway,xshift=0.8cm,yshift=0cm] {\Huge $m$} ; 
\draw [line width=0.03cm, decorate, decoration={brace, mirror, amplitude=10pt}](7.8,10) -- (7.8,13) node[black,midway,xshift=0.8cm,yshift=0cm] {\Huge $n$} ; 
\node at (2,0.5) {\Huge$a$} ; \node at (7.3,9.2) {\Huge$b$} ; 
\end{tikzpicture}
} }; 

\node at (6.7,0) {
\begin{tabular}{c||c}
weight & the number of weights \\ \hline
$(1,0)$& $\ell+m+n+2$ \\
$(0,1)$& $m$ \\
$(\mathchar`-1,0)$& $\ell+n+2$ \\
$(\mathchar`-1,\mathchar`-1)$& $m$ 
\end{tabular}
};
\end{tikzpicture}
\end{center}
\caption{The type (III) poset in Figure~\ref{figure_poset_Z2}, and the list of weights. Here, $\ell, m, n$ are the number of edges, where $\ell\ge 0, m\ge2, n\ge 0$.}
\label{poset_typeIII}
\end{figure}

\begin{figure}[H]
\begin{center}
\newcommand{\edgewidth}{0.06cm} 
\newcommand{\nodewidth}{0.06cm} 
\newcommand{\noderad}{0.2} 
\newcommand{\pmwidth}{0.65cm} 
\newcommand{\pmcolor}{lightgray} 

\begin{tikzpicture}
\node at (0,0)
{\scalebox{0.3}{
\begin{tikzpicture}
\coordinate (Min) at (4,0); \coordinate (Max) at (4,14);
\coordinate (N11) at (0,2); \coordinate (N12) at (0,5); \coordinate (N13) at (0,9); \coordinate (N14) at (0,12); 
\coordinate (N21) at (8,2); \coordinate (N22) at (8,5); \coordinate (N23) at (8,9); \coordinate (N24) at (8,12); 
\coordinate (N3) at (4,7);

\draw[line width=\pmwidth, \pmcolor]  (N11)--(N12)--(N3)--(N13)--(N14)--(Max) ; 
\draw[line width=\pmwidth, \pmcolor]  (Min)--(N21)--(N22)--(N3)--(N23)--(N24) ; 

\draw[line width=\edgewidth]  (Min)--(N11); \draw[line width=\edgewidth]  (N11)--(0,2.8); \draw[line width=\edgewidth, loosely dotted]  (0,3)--(0,4); 
\draw[line width=\edgewidth]  (0,4.2)--(N12); \draw[line width=\edgewidth]  (N12)--(N3); \draw[line width=\edgewidth]  (N3)--(N13); 
\draw[line width=\edgewidth]  (N13)--(0,9.8); \draw[line width=\edgewidth, loosely dotted]  (0,10)--(0,11); 
\draw[line width=\edgewidth]  (0,11.2)--(N14); \draw[line width=\edgewidth]  (N14)--(Max);
\draw[line width=\edgewidth]  (Min)--(N21); \draw[line width=\edgewidth]  (N21)--(8,2.8); \draw[line width=\edgewidth, loosely dotted]  (8,3)--(8,4); 
\draw[line width=\edgewidth]  (8,4.2)--(N22); \draw[line width=\edgewidth]  (N22)--(N3); \draw[line width=\edgewidth]  (N3)--(N23); 
\draw[line width=\edgewidth]  (N23)--(8,9.8); \draw[line width=\edgewidth, loosely dotted]  (8,10)--(8,11); 
\draw[line width=\edgewidth]  (8,11.2)--(N24); \draw[line width=\edgewidth]  (N24)--(Max);

\draw [line width=\nodewidth, fill=gray] (Min) circle [radius=\noderad] ; \draw [line width=\nodewidth, fill=gray] (Max) circle [radius=\noderad] ; 
\draw [line width=\nodewidth, fill=white] (N11) circle [radius=\noderad] ; \draw [line width=\nodewidth, fill=white] (N12) circle [radius=\noderad] ; 
\draw [line width=\nodewidth, fill=white] (N13) circle [radius=\noderad] ; \draw [line width=\nodewidth, fill=white] (N14) circle [radius=\noderad] ; 
\draw [line width=\nodewidth, fill=white] (N21) circle [radius=\noderad] ; \draw [line width=\nodewidth, fill=white] (N22) circle [radius=\noderad] ; 
\draw [line width=\nodewidth, fill=white] (N23) circle [radius=\noderad] ; \draw [line width=\nodewidth, fill=white] (N24) circle [radius=\noderad] ; 
\draw [line width=\nodewidth, fill=white] (N3) circle [radius=\noderad] ; 

\draw [line width=0.03cm, decorate, decoration={brace,amplitude=10pt}](-0.8,2) -- (-0.8,7) node[black,midway,xshift=-0.8cm,yshift=0cm] {\Huge $m$} ; 
\draw [line width=0.03cm, decorate, decoration={brace,amplitude=10pt}](-0.8,7) -- (-0.8,12) node[black,midway,xshift=-0.8cm,yshift=0cm] {\Huge $n$} ; 
\node at (2,0.5) {\Huge$a$} ; \node at (6,13.5) {\Huge$b$} ; 

\end{tikzpicture}
} }; 

\node at (6.5,0) {
\begin{tabular}{c||c}
weight & the number of weights \\ \hline
$(1,0)$& $m+1$ \\
$(0,1)$&  $n+1$ \\
$(\mathchar`-1,0)$& $m+1$ \\
$(0,\mathchar`-1)$& $n+1$ 
\end{tabular}
};
\end{tikzpicture}
\end{center}
\caption{The type (IV) poset in Figure~\ref{figure_poset_Z2}, and the list of weights. Here, $m, n$ are the number of edges, where $m\ge1, n\ge 1$.}
\label{poset_typeIV}
\end{figure}

\begin{figure}[H]
\begin{center}
\newcommand{\edgewidth}{0.06cm} 
\newcommand{\nodewidth}{0.06cm} 
\newcommand{\noderad}{0.2} 
\newcommand{\pmwidth}{0.65cm} 
\newcommand{\pmcolor}{lightgray} 

\begin{tikzpicture}
\node at (0,0)
{\scalebox{0.3}{
\begin{tikzpicture}

\coordinate (Min) at (4,0); \coordinate (Max) at (4,14);
\coordinate (N11) at (0,2); \coordinate (N12) at (0,3); \coordinate (N13) at (0,11); \coordinate (N14) at (0,12); 
\coordinate (N21) at (4,2); \coordinate (N22) at (4,3); \coordinate (N23) at (4,11); \coordinate (N24) at (4,12); 
\coordinate (N31) at (8,2); \coordinate (N32) at (8,3); \coordinate (N33) at (8,11); \coordinate (N34) at (8,12); 

\draw[line width=\pmwidth, \pmcolor]  (N11)--(N12)--(N13)--(N14)--(Max) ; 
\draw[line width=\pmwidth, \pmcolor]  (N21)--(N22)--(N23)--(N24)--(Max) ; 
\draw[line width=\pmwidth, \pmcolor]  (Min)--(N31)--(N32)--(N33)--(N34)--(Max) ; 

\draw[line width=\edgewidth]  (Min)--(N11); \draw[line width=\edgewidth]  (N11)--(N12); \draw[line width=\edgewidth]  (N12)--(0,3.8); 
\draw[line width=\edgewidth, loosely dotted]  (0,5)--(0,9); 
\draw[line width=\edgewidth]  (0,10.2)--(N13); \draw[line width=\edgewidth]  (N13)--(N14); \draw[line width=\edgewidth]  (N14)--(Max); 
\draw[line width=\edgewidth]  (Min)--(N21); \draw[line width=\edgewidth]  (N21)--(N22); \draw[line width=\edgewidth]  (N22)--(4,3.8); 
\draw[line width=\edgewidth, loosely dotted]  (4,5)--(4,9); 
\draw[line width=\edgewidth]  (4,10.2)--(N23); \draw[line width=\edgewidth]  (N23)--(N24); \draw[line width=\edgewidth]  (N24)--(Max); 
\draw[line width=\edgewidth]  (Min)--(N31); \draw[line width=\edgewidth]  (N31)--(N32); \draw[line width=\edgewidth]  (N32)--(8,3.8); 
\draw[line width=\edgewidth, loosely dotted]  (8,5)--(8,9); 
\draw[line width=\edgewidth]  (8,10.2)--(N33); \draw[line width=\edgewidth]  (N33)--(N34); \draw[line width=\edgewidth]  (N34)--(Max); 

\draw [line width=\nodewidth, fill=gray] (Min) circle [radius=\noderad] ; \draw [line width=\nodewidth, fill=gray] (Max) circle [radius=\noderad] ; 
\draw [line width=\nodewidth, fill=white] (N11) circle [radius=\noderad] ; \draw [line width=\nodewidth, fill=white] (N12) circle [radius=\noderad] ; 
\draw [line width=\nodewidth, fill=white] (N13) circle [radius=\noderad] ; \draw [line width=\nodewidth, fill=white] (N14) circle [radius=\noderad] ; 
\draw [line width=\nodewidth, fill=white] (N21) circle [radius=\noderad] ; \draw [line width=\nodewidth, fill=white] (N22) circle [radius=\noderad] ; 
\draw [line width=\nodewidth, fill=white] (N23) circle [radius=\noderad] ; \draw [line width=\nodewidth, fill=white] (N24) circle [radius=\noderad] ; 
\draw [line width=\nodewidth, fill=white] (N31) circle [radius=\noderad] ; \draw [line width=\nodewidth, fill=white] (N32) circle [radius=\noderad] ; 
\draw [line width=\nodewidth, fill=white] (N33) circle [radius=\noderad] ; \draw [line width=\nodewidth, fill=white] (N34) circle [radius=\noderad] ; 

\draw [line width=0.03cm, decorate, decoration={brace,amplitude=10pt}](-0.8,2) -- (-0.8,12) node[black,midway,xshift=-0.8cm,yshift=0cm] {\Huge $n$} ; 
\node at (2,0.5) {\Huge$a$} ; \node at (3.6,1.1) {\Huge$b$} ; 

\end{tikzpicture}
} }; 

\node at (6.5,0) {
\begin{tabular}{c||c}
weight & the number of weights \\ \hline
$(1,0)$& $n+2$ \\
$(0,1)$& $n+2$ \\
$(\mathchar`-1,\mathchar`-1)$& $n+2$ 
\end{tabular}
};
\end{tikzpicture}
\end{center}
\caption{The type (V) poset in Figure~\ref{figure_poset_Z2}, and the list of weights. Here, $n$ are the number of edges, where $n\ge 0$.}
\label{poset_typeV}
\end{figure}

\section{Rank one MCM modules for Hibi rings with class group $\ZZ^2$}
\label{sec_MCM_HibiZ2}

To complete the proof of Theorem~\ref{main_thm_HibiZ2}, in this section we give the explicit description of all rank one MCM modules for each type (I)--(V) shown in Figure~\ref{figure_poset_Z2}. 

\subsection{Van den Bergh's criterion of MCM modules} 

In order to check which rank one reflexive module is MCM, we will use the criterion given in \cite{VdB1}. 

\begin{notation}
First, we recall our settings. Let $R$ be a Gorenstein Hibi ring with $\Cl(R)\cong\ZZ^2$, which is obtained from one of the posets given in Figure~\ref{figure_poset_Z2}. 
For $G=\Hom(\Cl(R),\kk^\times)\cong(\kk^\times)^2$, $\beta_1,\cdots,\beta_n$ denote the weights of the characters corresponding to prime divisors on $\Spec R$, and we let $\calW=\{ 1,\cdots, n\}$. Then, $R\cong S^G$ under the action of $G$ defined via the weights $\beta_i$'s, where $S=\kk[x_1,\cdots,x_n]$. 
We note that it is known that $\Spec S$ contains a \emph{stable point} (i.e., a point in $\Spec S$ having the closed $G$-orbit and finite stabilizer) if and only if for any $0\neq\lambda\in\rmY(G)$ there exists a weight $\beta_i$ such that $\langle\lambda,\beta_i\rangle>0$. 
Therefore, we easily check that $\Spec S$ always contains a stable point in our situation. 

Let $\rmY(G)$ be the group of one-parameter subgroups of $G$, and let $\rmY(G)_\RR\coloneqq\rmY(G)\otimes_\ZZ\RR$. 
For $\lambda\in\rmY(G)_\RR$, we set $T_\lambda=\{i\in\calW \mid \langle\lambda,\beta_i\rangle<0\}$ and $T_\lambda^c=\{i\in\calW \mid \langle\lambda,\beta_i\rangle\ge0\}$. 
For $\lambda,\lambda^\prime\in\rmY(G)_\RR$, we denote $\lambda\sim\lambda^\prime$ if $T_\lambda=T_{\lambda^\prime}$. 
We also define 
$$
B(G)=\{\lambda\in\rmY(G)_\RR \mid \|\lambda\|<1\}, \quad B(G)_\lambda=\{\mu\in B(G)\mid \mu\sim\lambda\}, 
$$
and sometimes we simply denote these by $B, B_\lambda$ respectively. 
\end{notation}

For $0\neq\lambda\in B$, we see that $B_\lambda$ takes one of the following types: 
\begin{center}
\scalebox{0.8}{
\begin{tikzpicture}
\draw [line width=0.02cm] (-0.5,0.65)--(5,0.65) ; \draw [line width=0.02cm] (-0.5,-0.65)--(5,-0.65) ; 
\draw [line width=0.02cm] (-0.5,-1.95)--(5,-1.95) ; \draw [line width=0.02cm] (-0.5,-3.25)--(5,-3.25) ; 

\draw [line width=0.02cm] (-0.5,0.65)--(-0.5,-3.25) ; \draw [line width=0.02cm] (0.7,0.65)--(0.7,-3.25) ; 
\draw [line width=0.02cm] (5,0.65)--(5,-3.25) ; 

\node at (0,0) {$\Lambda_\emptyset$}; 
\node at (2,0) {\scalebox{0.5}{\begin{tikzpicture} \coordinate (A1) at (0:2.5cm); 
\draw [line width=0.05cm] (0,0)--(A1) ; \draw [line width=0.05cm, fill=white] (0,0) circle [radius=0.12cm] ; \draw [line width=0.05cm, fill=white] (A1) circle [radius=0.12cm] ;\end{tikzpicture} }}; 

\node at (4,0) {\scalebox{0.5}{\begin{tikzpicture} 
\coordinate (A2) at (-20:2.5cm); \coordinate (A3) at (20:2.5cm); 
\filldraw[fill=lightgray, draw=lightgray] (0,0)--(A2) arc [start angle=-20, delta angle=40, radius=2.5] ; 
\draw [line width=0.05cm] (0,0)--(A2) ; \draw [line width=0.05cm] (0,0)--(A3) ; 
\draw [line width=0.05cm, fill=white] (0,0) circle [radius=0.12cm] ; 
\draw [line width=0.05cm, dashed]  (A2) arc [start angle=-20, delta angle=40, radius=2.5] ;
\end{tikzpicture} }}; 

\node at (0,-1.3) {$\Lambda_\bullet$}; 
\node at (2,-1.3) 
{\scalebox{0.5}{\begin{tikzpicture} 
\filldraw[fill=lightgray, draw=lightgray] (0,0)--(A2) arc [start angle=-20, delta angle=40, radius=2.5] ; 
\draw [line width=0.05cm, dashed] (0,0)--(A2) ; \draw [line width=0.05cm] (0,0)--(A3) ; 
\draw [line width=0.05cm, fill=white] (0,0) circle [radius=0.12cm] ; 
\draw [line width=0.05cm, dashed]  (A2) arc [start angle=-20, delta angle=40, radius=2.5] ;
\end{tikzpicture} }}; 

\node at (4,-1.3) 
{\scalebox{0.5}{\begin{tikzpicture} 
\filldraw[fill=lightgray, draw=lightgray] (0,0)--(A2) arc [start angle=-20, delta angle=40, radius=2.5] ; 
\draw [line width=0.05cm] (0,0)--(A2) ; \draw [line width=0.05cm, dashed] (0,0)--(A3) ; 
\draw [line width=0.05cm, fill=white] (0,0) circle [radius=0.12cm] ; 
\draw [line width=0.05cm, dashed]  (A2) arc [start angle=-20, delta angle=40, radius=2.5] ;
\end{tikzpicture} }}; 

\node at (0,-2.6) {$\Lambda_{\bullet\bullet}$}; 
\node at (2,-2.6)
{\scalebox{0.5}{\begin{tikzpicture} 
\filldraw[fill=lightgray, draw=lightgray] (0,0)--(A2) arc [start angle=-20, delta angle=40, radius=2.5] ; 
\draw [line width=0.05cm, dashed] (0,0)--(A2) ; \draw [line width=0.05cm, dashed] (0,0)--(A3) ; 
\draw [line width=0.05cm, fill=white] (0,0) circle [radius=0.12cm] ; 
\draw [line width=0.05cm, dashed]  (A2) arc [start angle=-20, delta angle=40, radius=2.5] ;
\end{tikzpicture} }}; 
\end{tikzpicture}
}
\end{center}
Following \cite{VdB1}, we denote the set of representative elements $0\neq\bar{\lambda}\in(B/{\sim})$ corresponding to each type by $\Lambda_\emptyset$, $\Lambda_\bullet$, and $\Lambda_{\bullet\bullet}$ as in the above table (see \cite[Corollary~4.1.1]{VdB1} for more precise descriptions). 
Then, we have the following criterion. 

\begin{proposition}[{see \cite[Remark~4.1.2]{VdB1}}] 
\label{criterion_MCM}
Let $\calU^\chi\subset\ZZ^n$ be the set of all integral solutions $\bfa=(a_1,\cdots,a_n)$ to $\chi=\sum_{i=1}^na_i\beta_i\in\rmX(G)$, and let ${\rm supp}_-\bfa\coloneqq\{i \mid a_i<0\}$ for $\bfa\in\calU^\chi$. 
We suppose that $|T_\lambda|>1$ for all $\bar{\lambda}\in\Lambda_\emptyset$ and $|T_\lambda|>2$ for all $\bar{\lambda}\in\Lambda_{\bullet\bullet}$. 
Then, $M_\chi$ is MCM if and only if for any $\bar{\lambda}\in\Lambda_\emptyset\cup\Lambda_{\bullet\bullet}$, there is no $\bfa\in\calU^\chi$ such that ${\rm supp}_-\bfa=T_\lambda^c$. 

\end{proposition}

\subsection{Rank one MCM modules for the type (I)}
\begin{proposition}
\label{prop_MCM_typeI}
Let $R$ be the Gorenstein Hibi ring with $\Cl(R)\cong\ZZ^2$ associated with the poset shown in Figure~\ref{poset_typeI}. 
Then, we see that $M_\chi$ is a rank one MCM module if and only if $\chi\in\rmX(G)\cong\ZZ^2$ is contained in the shaded area in Figure~\ref{MCM_typeI}. 
In particular, the grayed area represents conic classes. 
\end{proposition}

\begin{figure}[H]
\begin{center}
\newcommand{\edgewidth}{0.06cm} 
\newcommand{\nodewidth}{0.06cm} 
\newcommand{\noderad}{0.15} 
\newcommand{\pmwidth}{0.65cm} 
\newcommand{\pmcolor}{lightgray} 

{\scalebox{0.35}{
\begin{tikzpicture}

\filldraw [line width=\edgewidth, fill=lightgray] (6,4)--(6,0)--(2,-4)--(-6,-4)--(-6,0)--(-2,4)--(6,4) ; 
\draw[pattern=my north west lines, pattern color=blue] (10,4)--(6,0)--(6,-4)--(-10,-4)--(-6,0)--(-6,4)--(10,4) ; 
\draw [line width=\edgewidth] (6,4)--(6,-4)--(-6,-4)--(-6,4)--(6,4) ;
\draw [line width=\edgewidth] (6,4)--(10,4)--(6,0) ; \draw [line width=\edgewidth] (-6,-4)--(-10,-4)--(-6,0) ;
\draw[->, line width=0.085cm]  (-10,0)--(10,0); \draw[->, line width=0.085cm]  (0,-6)--(0,6); 

\node at (1,4.7) {\huge $(0,n)$} ; \node at (-1.2,-4.7) {\huge $(0,-n)$} ; 
\node at (8.25,-0.65) {\huge $(m+n+1,0)$} ; \node at (-8.6,0.65) {\huge $(-m-n-1,0)$} ; 
\node at (3.6,-4.7) {\huge $(m+1,-n)$} ; \node at (-3.6,4.7) {\huge $(-m-1,n)$} ; 

\node at (10,4.65) {\huge $(m+2n+1,n)$} ; \node at (-10,-4.65) {\huge $(-m-2n-1,-n)$} ; 

\draw [line width=\nodewidth, fill=black] (0,4) circle [radius=\noderad] ; \draw [line width=\nodewidth, fill=black] (0,-4) circle [radius=\noderad] ; 
\draw [line width=\nodewidth, fill=black] (6,0) circle [radius=\noderad] ; \draw [line width=\nodewidth, fill=black] (-6,0) circle [radius=\noderad] ; 
\draw [line width=\nodewidth, fill=black] (2,-4) circle [radius=\noderad] ; \draw [line width=\nodewidth, fill=black] (-2,4) circle [radius=\noderad] ; 
\draw [line width=\nodewidth, fill=black] (10,4) circle [radius=\noderad] ; \draw [line width=\nodewidth, fill=black] (-10,-4) circle [radius=\noderad] ; 
\end{tikzpicture}
} }
\end{center}
\caption{The region of rank one MCM modules for type (I)}
\label{MCM_typeI}
\end{figure}
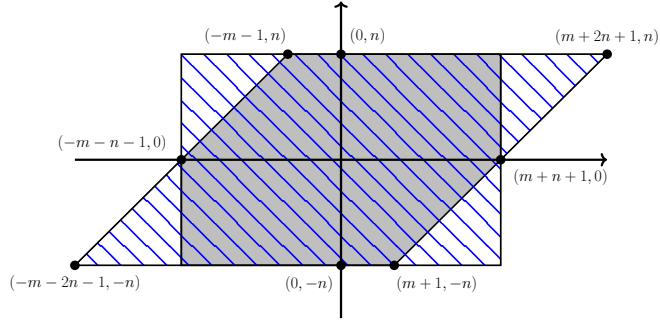

\begin{proof}
By Theorem~\ref{conic_Hibi}, $\chi=(c_1,c_2)$ represents a conic module if and only if $\chi\in\calC(P)$ where 
\begin{align*}
\calC(P)=\{\chi=(c_1,c_2)\mid -(m+n+1)\le c_1\le m+n+1, &-n\le c_2\le n,\\ &-(m+n+1)\le c_1-c_2\le m+n+1\}. 
\end{align*}
Thus, for any character $\chi$ contained in $\calC(P)$, we see that $M_\chi$ is a conic module, and hence it is MCM. 

Then, we consider rank one MCM modules that are not conic. 
In order to use Proposition~\ref{criterion_MCM}, we decompose $B{\setminus}\{0\}\subset\rmY(G)_\RR$ into the types $\Lambda_\emptyset$, $\Lambda_\bullet$, and $\Lambda_{\bullet\bullet}$ as shown in the following figure. 

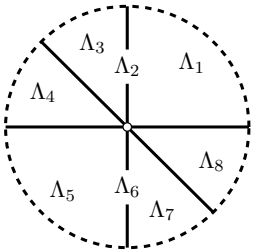
\begin{center}

\begin{tikzpicture}
\node at (0,0) {
\scalebox{0.8}{\begin{tikzpicture} 
\draw [line width=0.05cm] (0,0)--(0:2cm) ; \draw [line width=0.05cm] (0,0)--(135:2cm) ; \draw [line width=0.05cm] (0,0)--(180:2cm) ; 
\draw [line width=0.05cm] (0,0)--(315:2cm) ; 
\draw [line width=0.05cm, fill=white] (0,0)--(90:2cm) node[fill=white,inner sep=0.3pt, circle, midway,xshift=0cm,yshift=0cm] {\large$\Lambda_2$} ; 
\draw [line width=0.05cm, fill=white] (0,0)--(270:2cm) node[fill=white,inner sep=0.3pt, circle, midway,xshift=0cm,yshift=0cm] {\large$\Lambda_6$} ; 
\draw [line width=0.02cm, fill=white] (0,0) circle [radius=0.07cm] ; \draw [line width=0.05cm, dashed] (0,0) circle [radius=2cm] ; 

\coordinate (C1) at (45:1.5cm); \node at (C1) {\large$\Lambda_1$}; 
\coordinate (C3) at (112.5:1.5cm); \node at (C3) {\large$\Lambda_3$}; 
\coordinate (C4) at (157.5:1.5cm); \node at (C4) {\large$\Lambda_4$}; 
\coordinate (C5) at (225:1.5cm); \node at (C5) {\large$\Lambda_5$}; 
\coordinate (C7) at (292.5:1.5cm); \node at (C7) {\large$\Lambda_7$}; 
\coordinate (C8) at (337.5:1.5cm); \node at (C8) {\large$\Lambda_8$}; 
\end{tikzpicture}
}};

\node at (7.5,0) {
\scalebox{0.8}{
\begin{tabular}{c|l|c|l}
&type &weight $\beta_i$ satisfying $\langle\lambda,\beta_i\rangle<0$ for $\lambda\in\Lambda_i$&multiplicity \\ \hline
$\Lambda_1$&$\Lambda_\bullet$& $(-1,0)$&$m+1$ \\
&&$(-1,-1)$&$n+1$\\ \hline
$\Lambda_2$&$\Lambda_\emptyset$& $(-1,-1)$&$n+1$ \\ \hline
$\Lambda_3$&$\Lambda_{\bullet\bullet}$& $(1,0)$&$m+n+2$ \\
&&$(-1,-1)$&$n+1$ \\ \hline
$\Lambda_4$&$\Lambda_\emptyset$& $(1,0)$&$m+n+2$ \\ \hline
$\Lambda_5$&$\Lambda_{\bullet\bullet}$& $(1,0)$&$m+n+2$ \\
&&$(0,1)$&$n+1$ \\ \hline
$\Lambda_6$&$\Lambda_\emptyset$& $(0,1)$&$n+1$ \\ \hline
$\Lambda_7$&$\Lambda_\bullet$& $(0,1)$&$n+1$ \\ 
&&$(-1,0)$&$m+1$ \\ \hline
$\Lambda_8$&$\Lambda_{\bullet\bullet}$& $(0,1)$&$n+1$ \\
&&$(-1,0)$&$m+1$ \\ 
&&$(-1,-1)$&$n+1$ 
\end{tabular}
}};
\end{tikzpicture} 
\end{center}

As we saw in Proposition~\ref{criterion_MCM}, we need not consider the region with the type $\Lambda_\bullet$ for determining rank one MCM modules. 
Thus, we first pay attention to the region $\Lambda_2$ which is the type $\Lambda_\emptyset$ and the weights $\beta_i$ satisfying $\langle\lambda,\beta_i\rangle<0$ for $\lambda\in\Lambda_2$ are $(-1,-1)$ with the multiplicity $n+1$. 
By Proposition~\ref{criterion_MCM}, $M_\chi$ is not MCM if and only if 
there exists integers $\bfa=(a_i)$ satisfying $\chi=\sum a_i\beta_i$ with ${\rm supp}_-\bfa=T_\lambda^c$ for some $\bar{\lambda}\in\Lambda_\emptyset\cup\Lambda_{\bullet\bullet}$, 
and such a character $\chi$ can be described as 
\begin{equation}
\label{nonMCM_character}
\chi=\sum_{s\in T_\lambda^c}a_s\beta_s+\sum_{t\in T_\lambda}a_t\beta_t\in\rmX(G)
\end{equation}
where $a_s\in\ZZ_{<0}$ and $a_t\in\ZZ_{\ge 0}$. 
For $\lambda\in\Lambda_2$, we can write this as 
$$
\chi=\sum_{h=1}^{m+n+2}(a_h,0)+\sum_{i=1}^{n+1}(0,a_i)+\sum_{j=1}^{m+1}(-a_j,0)+\sum_{k=1}^{n+1}(-a_k,-a_k)
$$
with $a_h,a_i,a_j\in\ZZ_{<0}$ and $a_k\in\ZZ_{\ge 0}$ (see the table given in Figure~\ref{poset_typeI}). 
Thus, we see that such characters are contained in the grayed area in the left of the following figure, and the corresponding modules of covariants are not MCM. 

\begin{center}
\begin{tikzpicture}
\node at (0,0) 
{\scalebox{0.7}{
\begin{tikzpicture} 
\filldraw[fill=lightgray, draw=lightgray] (-2.5,0) rectangle (2.5,-2.5) ; 
\draw [->,line width=0.05cm,red] (0,0)--(1,0) ; \draw [->, line width=0.05cm,red] (0,0)--(0,-1) ; \draw [->, line width=0.05cm,red] (0,0)--(-1,-1) ; 
\draw [->,line width=0.05cm,red] (0,0)--(-1,0) ; 
\draw [line width=0.05cm, fill=black] (0,0) circle [radius=0.08cm] ; 
\node at (0.5,0.4) {$(-n-1, -n-1)$}; 
\node[red] at (1.3,-0.35) {$(1,0)$} ; \node[red] at (0.35,-1.3) {$(0,\mathchar`-1)$} ; \node[red] at (-1.5,-1.2) {$(\mathchar`-1,\mathchar`-1)$} ; 
\node[red] at (-1.35,-0.35) {$(\mathchar`-1,0)$} ; 
\node at (1.6,-2.3) {not MCM}; 
\end{tikzpicture} 
}};

\node at (6.5,0) 
{\scalebox{0.7}{
\begin{tikzpicture} 
\filldraw[fill=lightgray, draw=lightgray] (0,0)--(2.5,0)--(2.5,-2.5)--(-2.5,-2.5)--(0,0) ; 
\draw [->,line width=0.05cm,red] (0,0)--(1,0) ; \draw [->, line width=0.05cm,red] (0,0)--(0,-1) ; \draw [->, line width=0.05cm,red] (0,0)--(-1,-1) ; 
\draw [line width=0.05cm, fill=black] (0,0) circle [radius=0.08cm] ; 
\node at (1,0.4) {$(m+1, -n-1)$}; 
\node[red] at (1.3,-0.35) {$(1,0)$} ; \node[red] at (0.35,-1.3) {$(0,\mathchar`-1)$} ; \node[red] at (-1.6,-0.7) {$(\mathchar`-1,\mathchar`-1)$} ; 
\node at (1.6,-2.3) {not MCM}; 
\end{tikzpicture} 
}};

\end{tikzpicture} 
\end{center}

Similarly, we consider the region $\Lambda_3$ which is the type $\Lambda_{\bullet\bullet}$.  The weights $\beta_i$ satisfying $\langle\lambda,\beta_i\rangle<0$ for $\lambda\in\Lambda_3$ are 
$(1,0)$ with the multiplicity $(m+n+2)$ and $(-1,-1)$ with the multiplicity $n+1$. 
Then, for $\lambda\in\Lambda_3$, we consider a character $\chi$ with the form (\ref{nonMCM_character}). 
Such character can be described as 
$$
\chi=\sum_{i=1}^{n+1}(0,a_i)+\sum_{j=1}^{m+1}(-a_j,0)+\sum_{h=1}^{m+n+2}(a_h,0)+\sum_{k=1}^{n+1}(-a_k,-a_k)
$$
with $a_i,a_j\in\ZZ_{<0}$ and $a_h,a_k\in\ZZ_{\ge 0}$. 
Thus, we see that these characters are contained in the grayed area in the right of the above figure, and the corresponding modules of covariants are not MCM. 

We repeat these arguments for the regions $\Lambda_4$, $\Lambda_5$, $\Lambda_6$, and $\Lambda_8$ which have the type $\Lambda_\emptyset$ or $\Lambda_{\bullet\bullet}$. 
Then, we have the region of characters corresponding to rank one MCM modules as shown in Figure~\ref{MCM_typeI}. 
\end{proof}

\subsection{Rank one MCM modules for the type (II)}

\begin{proposition}
\label{prop_MCM_typeII}
Let $R$ be the Gorenstein Hibi ring with $\Cl(R)\cong\ZZ^2$ associated with the poset shown in Figure~\ref{poset_typeII}. 
Then, we see that $M_\chi$ is a rank one MCM module if and only if $\chi\in\rmX(G)\cong\ZZ^2$ is contained in the shaded area in Figure~\ref{MCM_typeII}. 
In particular, the grayed area represents conic classes. 
\end{proposition}

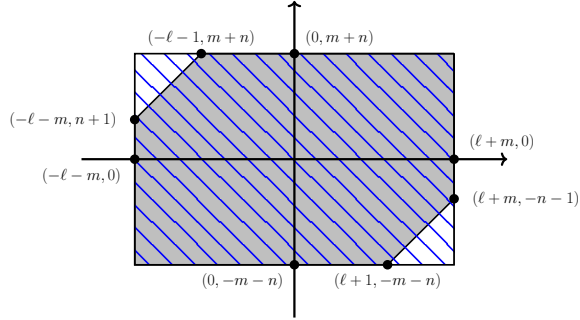
\begin{figure}[H]
\begin{center}
\newcommand{\edgewidth}{0.06cm} 
\newcommand{\nodewidth}{0.06cm} 
\newcommand{\noderad}{0.15} 
\newcommand{\pmwidth}{0.65cm} 
\newcommand{\pmcolor}{lightgray} 

{\scalebox{0.35}{
\begin{tikzpicture}

\filldraw [line width=\edgewidth, fill=lightgray] (6,4)--(6,-1.5)--(3.5,-4)--(-6,-4)--(-6,1.5)--(-3.5,4)--(6,4) ; 
\draw[pattern=my north west lines, pattern color=blue] (-6,-4) rectangle (6,4);
\draw [line width=\edgewidth] (6,4)--(6,-4)--(-6,-4)--(-6,4)--(6,4) ;
\draw[->, line width=0.085cm]  (-8,0)--(8,0); \draw[->, line width=0.085cm]  (0,-6)--(0,6); 

\node at (1.7,4.6) {\huge $(0,m+n)$} ; \node at (-1.95,-4.6) {\huge $(0,-m-n)$} ; 
\node at (7.8,0.65) {\huge $(\ell+m,0)$} ; \node at (-8,-0.65) {\huge $(-\ell-m,0)$} ; 

\node at (8.7,-1.5) {\huge $(\ell+m,-n-1)$} ; \node at (-8.7,1.5) {\huge $(-\ell-m,n+1)$} ; 
\node at (3.5,-4.6) {\huge $(\ell+1,-m-n)$} ; \node at (-3.5,4.6) {\huge $(-\ell-1,m+n)$} ; 

\draw [line width=\nodewidth, fill=black] (0,4) circle [radius=\noderad] ; \draw [line width=\nodewidth, fill=black] (0,-4) circle [radius=\noderad] ; 
\draw [line width=\nodewidth, fill=black] (6,0) circle [radius=\noderad] ; \draw [line width=\nodewidth, fill=black] (-6,0) circle [radius=\noderad] ; 
\draw [line width=\nodewidth, fill=black] (6,-1.5) circle [radius=\noderad] ; \draw [line width=\nodewidth, fill=black] (3.5,-4) circle [radius=\noderad] ; 
\draw [line width=\nodewidth, fill=black] (-6,1.5) circle [radius=\noderad] ; \draw [line width=\nodewidth, fill=black] (-3.5,4) circle [radius=\noderad] ; 

\end{tikzpicture}
} }
\end{center}
\caption{The region of rank one MCM modules for type (II)}
\label{MCM_typeII}
\end{figure}

\begin{proof}
By Theorem~\ref{conic_Hibi}, $\chi=(c_1,c_2)$ represents a conic module if and only if $\chi\in\calC(P)$ where 
\begin{align*}
\calC(P)=\{\chi=(c_1,c_2)\mid -(\ell+m)\le c_1\le\ell+m, &-(m+n)\le c_2\le m+n,\\ &-(\ell+m+n+1)\le c_1-c_2\le\ell+m+n+1\}. 
\end{align*}

The remaining assertion follows from the decomposition of $B{\setminus}\{0\}$ shown below, and a similar argument as in the proof of Proposition~\ref{prop_MCM_typeI}.

\begin{center}

\begin{tikzpicture}
\node at (0,0) {
\scalebox{0.8}{\begin{tikzpicture} 
\draw [line width=0.05cm] (0,0)--(0:2cm) ; 
\draw [line width=0.05cm] (0,0)--(135:2cm) ; \draw [line width=0.05cm] (0,0)--(180:2cm) ; 
\draw [line width=0.05cm] (0,0)--(315:2cm) ; 
\draw [line width=0.05cm, fill=white] (0,0)--(90:2cm) node[fill=white,inner sep=0.3pt, circle, midway,xshift=0cm,yshift=0cm] {\large$\Lambda_2$} ; 
\draw [line width=0.05cm, fill=white] (0,0)--(180:2cm) node[fill=white,inner sep=0.3pt, circle, midway,xshift=0cm,yshift=0cm] {\large$\Lambda_5$} ; 
\draw [line width=0.05cm, fill=white] (0,0)--(270:2cm) node[fill=white,inner sep=0.3pt, circle, midway,xshift=0cm,yshift=0cm] {\large$\Lambda_7$} ; 
\draw [line width=0.05cm, fill=white] (0,0)--(0:2cm) node[fill=white,inner sep=0.3pt, circle, midway,xshift=0cm,yshift=0cm] {\large$\Lambda_{10}$} ; 
\draw [line width=0.02cm, fill=white] (0,0) circle [radius=0.07cm] ; \draw [line width=0.05cm, dashed] (0,0) circle [radius=2cm] ; 

\coordinate (C1) at (45:1.5cm); \node at (C1) {\large$\Lambda_1$}; 
\coordinate (C3) at (112.5:1.5cm); \node at (C3) {\large$\Lambda_3$}; 
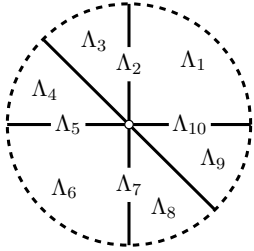
\coordinate (C4) at (157.5:1.5cm); \node at (C4) {\large$\Lambda_4$}; 
\coordinate (C6) at (225:1.5cm); \node at (C6) {\large$\Lambda_6$}; 

\coordinate (C8) at (292.5:1.5cm); \node at (C8) {\large$\Lambda_8$}; 
\coordinate (C9) at (337.5:1.5cm); \node at (C9) {\large$\Lambda_9$}; 
\end{tikzpicture}
}};

\node at (7.5,0) {
\scalebox{0.8}{
\begin{tabular}{c|l|c|l}
&type &weight $\beta_i$ satisfying $\langle\lambda,\beta_i\rangle<0$ for $\lambda\in\Lambda_i$&multiplicity\\ \hline
$\Lambda_1$&$\Lambda_{\bullet\bullet}$& $(0,-1)$&$n+1$ \\
&&$(-1,-1)$&$m$ \\ 
&&$(-1,0)$&$\ell+1$ \\ \hline
$\Lambda_2$&$\Lambda_\emptyset$& $(0,-1)$&$n+1$ \\ 
&&$(-1,-1)$&$m$ \\ \hline
$\Lambda_3$&$\Lambda_{\bullet\bullet}$&$(0,-1)$&$n+1$ \\ 
&&$(-1,-1)$&$m$ \\
&&$(1,0)$&$\ell+m+1$ \\ \hline 
$\Lambda_4$&$\Lambda_\bullet$&$(0,-1)$&$n+1$ \\
&&$(1,0)$&$\ell+m+1$ \\ \hline
$\Lambda_5$&$\Lambda_\emptyset$& $(1,0)$&$\ell+m+1$ \\ \hline
$\Lambda_6$&$\Lambda_{\bullet\bullet}$&$(1,0)$&$\ell+m+1$ \\ 
&&$(0,1)$&$m+n+1$ \\ \hline
$\Lambda_7$&$\Lambda_\emptyset$& $(0,1)$&$m+n+1$ \\ \hline
$\Lambda_8$&$\Lambda_\bullet$& $(0,1)$&$m+n+1$ \\ 
&&$(-1,0)$&$\ell+1$ \\ \hline
$\Lambda_9$&$\Lambda_{\bullet\bullet}$&$(0,1)$&$m+n+1$ \\ 
&&$(-1,0)$&$\ell+1$ \\ 
&&$(-1,-1)$&$m$ \\ \hline
$\Lambda_{10}$&$\Lambda_\emptyset$&$(-1,0)$&$\ell+1$ \\
&&$(-1,-1)$&$m$ 
\end{tabular}
}};

\end{tikzpicture} 
\end{center}
\end{proof}

\subsection{Rank one MCM modules for the type (III)}

\begin{proposition}
\label{prop_MCM_typeIII}
Let $R$ be the Gorenstein Hibi ring with $\Cl(R)\cong\ZZ^2$ associated with the poset shown in Figure~\ref{poset_typeIII}. 
Then, we see that $M_\chi$ is a rank one MCM module if and only if $\chi\in\rmX(G)\cong\ZZ^2$ is contained in the shaded area in Figure~\ref{MCM_typeIII}. 
In particular, the grayed area represents conic classes. 
\end{proposition}

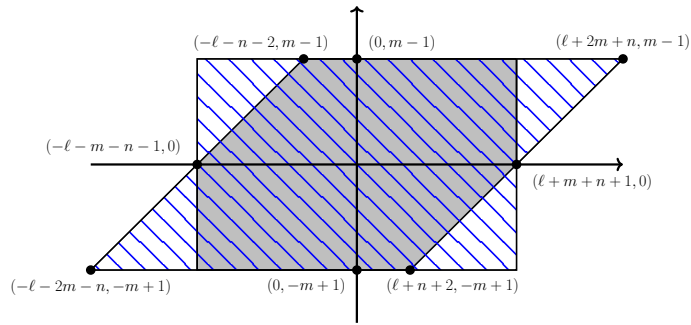
\begin{figure}[H]
\begin{center}
\newcommand{\edgewidth}{0.06cm} 
\newcommand{\nodewidth}{0.06cm} 
\newcommand{\noderad}{0.15} 
\newcommand{\pmwidth}{0.65cm} 
\newcommand{\pmcolor}{lightgray} 

{\scalebox{0.35}{
\begin{tikzpicture}

\filldraw [line width=\edgewidth, fill=lightgray] (6,4)--(6,0)--(2,-4)--(-6,-4)--(-6,0)--(-2,4)--(6,4) ; 
\draw[pattern=my north west lines, pattern color=blue] (10,4)--(6,0)--(6,-4)--(-10,-4)--(-6,0)--(-6,4)--(10,4) ; 
\draw [line width=\edgewidth] (6,4)--(6,-4)--(-6,-4)--(-6,4)--(6,4) ;
\draw [line width=\edgewidth] (6,4)--(10,4)--(6,0) ; \draw [line width=\edgewidth] (-6,-4)--(-10,-4)--(-6,0) ;
\draw[->, line width=0.085cm]  (-10,0)--(10,0); \draw[->, line width=0.085cm]  (0,-6)--(0,6); 

\node at (1.7,4.6) {\huge $(0,m-1)$} ; \node at (-1.9,-4.6) {\huge $(0,-m+1)$} ; 
\node at (8.85,-0.65) {\huge $(\ell+m+n+1,0)$} ; \node at (-9.15,0.65) {\huge $(-\ell-m-n-1,0)$} ; 
\node at (3.6,-4.6) {\huge $(\ell+n+2,-m+1)$} ; \node at (-3.6,4.6) {\huge $(-\ell-n-2,m-1)$} ; 

\node at (10,4.65) {\huge $(\ell+2m+n,m-1)$} ; \node at (-10,-4.65) {\huge $(-\ell-2m-n,-m+1)$} ; 

\draw [line width=\nodewidth, fill=black] (0,4) circle [radius=\noderad] ; \draw [line width=\nodewidth, fill=black] (0,-4) circle [radius=\noderad] ; 
\draw [line width=\nodewidth, fill=black] (6,0) circle [radius=\noderad] ; \draw [line width=\nodewidth, fill=black] (-6,0) circle [radius=\noderad] ; 
\draw [line width=\nodewidth, fill=black] (2,-4) circle [radius=\noderad] ; \draw [line width=\nodewidth, fill=black] (-2,4) circle [radius=\noderad] ; 
\draw [line width=\nodewidth, fill=black] (10,4) circle [radius=\noderad] ; \draw [line width=\nodewidth, fill=black] (-10,-4) circle [radius=\noderad] ; 
\end{tikzpicture}
} }
\end{center}
\caption{The region of rank one MCM modules for type (III)}
\label{MCM_typeIII}
\end{figure}

\begin{proof}
By Theorem~\ref{conic_Hibi}, $\chi=(c_1,c_2)$ represents a conic module if and only if $\chi\in\calC(P)$ where 
\begin{align*}
\calC(P)=\{\chi=(c_1,c_2)\mid -(m-1)\le c_2\le m+1, & -(\ell+m+n+1)\le c_1\le \ell+m+n+1, \\ &-(\ell+m+n+1)\le c_1-c_2\le \ell+m+n+1\}. 
\end{align*}

The remaining assertion follows from the decomposition of $B{\setminus}\{0\}$ shown below, and a similar argument as in the proof of Proposition~\ref{prop_MCM_typeI}.

\begin{center}

\begin{tikzpicture}
\node at (0,0) {
\scalebox{0.8}{\begin{tikzpicture} 
\draw [line width=0.05cm] (0,0)--(0:2cm) ; 
\draw [line width=0.05cm] (0,0)--(135:2cm) ; \draw [line width=0.05cm] (0,0)--(180:2cm) ; 
\draw [line width=0.05cm] (0,0)--(315:2cm) ; 
\draw [line width=0.05cm, fill=white] (0,0)--(90:2cm) node[fill=white,inner sep=0.3pt, circle, midway,xshift=0cm,yshift=0cm] {\large$\Lambda_2$} ; 
\draw [line width=0.05cm, fill=white] (0,0)--(270:2cm) node[fill=white,inner sep=0.3pt, circle, midway,xshift=0cm,yshift=0cm] {\large$\Lambda_6$} ; 
\draw [line width=0.02cm, fill=white] (0,0) circle [radius=0.07cm] ; \draw [line width=0.05cm, dashed] (0,0) circle [radius=2cm] ; 

\coordinate (C1) at (45:1.5cm); \node at (C1) {\large$\Lambda_1$}; 
\coordinate (C3) at (112.5:1.5cm); \node at (C3) {\large$\Lambda_3$}; 
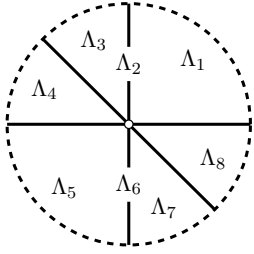
\coordinate (C4) at (157.5:1.5cm); \node at (C4) {\large$\Lambda_4$}; 
\coordinate (C5) at (225:1.5cm); \node at (C5) {\large$\Lambda_5$}; 
\coordinate (C7) at (292.5:1.5cm); \node at (C7) {\large$\Lambda_7$}; 
\coordinate (C8) at (337.5:1.5cm); \node at (C8) {\large$\Lambda_8$}; 
\end{tikzpicture}
}};

\node at (7.5,0) {
\scalebox{0.8}{
\begin{tabular}{c|l|c|l}
&type &weight $\beta_i$ satisfying $\langle\lambda,\beta_i\rangle<0$ for $\lambda\in\Lambda_i$&multiplicity\\ \hline
$\Lambda_1$&$\Lambda_\bullet$& $(-1,-1)$&$m$ \\ 
&&$(-1,0)$&$\ell+n+2$ \\ \hline
$\Lambda_2$&$\Lambda_\emptyset$&$(-1,-1)$&$m$ \\ \hline
$\Lambda_3$&$\Lambda_{\bullet\bullet}$&$(-1,-1)$&$m$ \\
&&$(1,0)$&$\ell+m+n+2$ \\ \hline
$\Lambda_4$&$\Lambda_\emptyset$&$(1,0)$&$\ell+m+n+2$ \\ \hline
$\Lambda_5$&$\Lambda_{\bullet\bullet}$&$(1,0)$&$\ell+m+n+2$ \\ 
&&$(0,1)$&$m$ \\ \hline
$\Lambda_6$&$\Lambda_\emptyset$&$(0,1)$&$m$ \\ \hline
$\Lambda_7$&$\Lambda_\bullet$&$(0,1)$&$m$ \\  
&&$(-1,0)$&$\ell+n+2$ \\ \hline
$\Lambda_8$&$\Lambda_{\bullet\bullet}$&$(0,1)$&$m$ \\ 
&&$(-1,0)$&$\ell+n+2$ \\ 
&&$(-1,-1)$&$m$
\end{tabular}
}};

\end{tikzpicture} 
\end{center}
\end{proof}

\subsection{Rank one MCM modules for the type (IV)}

\begin{proposition}
\label{prop_MCM_typeIV}
Let $R$ be the Gorenstein Hibi ring with $\Cl(R)\cong\ZZ^2$ associated with the poset shown in Figure~\ref{poset_typeIV}. 
Then, we see that $M_\chi$ is a rank one MCM module if and only if $\chi\in\rmX(G)\cong\ZZ^2$ is contained in the shaded area in Figure~\ref{MCM_typeIV}. 
In particular, the grayed area represents conic classes. 
(In this case, rank one MCM modules are precisely conic ones.)
\end{proposition}

\begin{figure}[H]
\begin{center}
\newcommand{\edgewidth}{0.06cm} 
\newcommand{\nodewidth}{0.06cm} 
\newcommand{\noderad}{0.15} 
\newcommand{\pmwidth}{0.65cm} 
\newcommand{\pmcolor}{lightgray} 

{\scalebox{0.35}{
\begin{tikzpicture}

\filldraw [line width=\edgewidth, fill=lightgray] (6,4)--(-6,4)--(-6,-4)--(6,-4)--(6,4) ; 
\draw[pattern=my north west lines, pattern color=blue] (6,4)--(-6,4)--(-6,-4)--(6,-4)--(6,4) ; 
\draw[->, line width=0.085cm]  (-8,0)--(8,0); \draw[->, line width=0.085cm]  (0,-6)--(0,6); 

\node at (1,4.5) {\huge $(0,n)$} ; \node at (-1.2,-4.5) {\huge $(0,-n)$} ; 
\node at (7.1,-0.65) {\huge $(m,0)$} ; \node at (-7.35,0.65) {\huge $(-m,0)$} ; 

\draw [line width=\nodewidth, fill=black] (0,4) circle [radius=\noderad] ; \draw [line width=\nodewidth, fill=black] (0,-4) circle [radius=\noderad] ; 
\draw [line width=\nodewidth, fill=black] (6,0) circle [radius=\noderad] ; \draw [line width=\nodewidth, fill=black] (-6,0) circle [radius=\noderad] ; 
\end{tikzpicture}
} }
\end{center}
\caption{The region of rank one MCM modules for type (IV)}
\label{MCM_typeIV}
\end{figure}
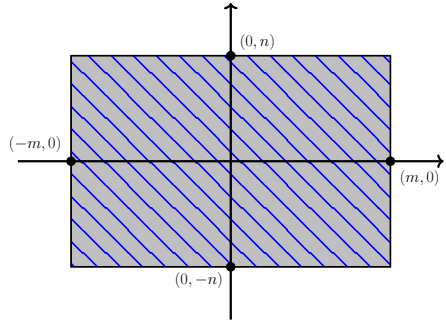

\begin{proof}
By Theorem~\ref{conic_Hibi}, $\chi=(c_1,c_2)$ represents a conic module if and only if $\chi\in\calC(P)$ where 
$$
\calC(P)=\{\chi=(c_1,c_2)\mid -m\le c_1\le m, -n\le c_2\le n \}. 
$$

The remaining assertion follows from the decomposition of $B{\setminus}\{0\}$ shown below, and a similar argument as in the proof of Proposition~\ref{prop_MCM_typeI}.

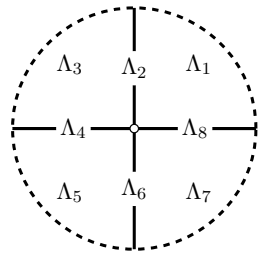
\begin{center}

\begin{tikzpicture}
\node at (0,0) {
\scalebox{0.8}{\begin{tikzpicture} 
\draw [line width=0.05cm] (0,0)--(0:2cm) ; 
\draw [line width=0.05cm] (0,0)--(180:2cm) ; 

\draw [line width=0.05cm, fill=white] (0,0)--(90:2cm) node[fill=white,inner sep=0.3pt, circle, midway,xshift=0cm,yshift=0cm] {\large$\Lambda_2$} ; 
\draw [line width=0.05cm, fill=white] (0,0)--(180:2cm) node[fill=white,inner sep=0.3pt, circle, midway,xshift=0cm,yshift=0cm] {\large$\Lambda_4$} ; 
\draw [line width=0.05cm, fill=white] (0,0)--(270:2cm) node[fill=white,inner sep=0.3pt, circle, midway,xshift=0cm,yshift=0cm] {\large$\Lambda_6$} ; 
\draw [line width=0.05cm, fill=white] (0,0)--(0:2cm) node[fill=white,inner sep=0.3pt, circle, midway,xshift=0cm,yshift=0cm] {\large$\Lambda_8$} ; 
\draw [line width=0.02cm, fill=white] (0,0) circle [radius=0.07cm] ; \draw [line width=0.05cm, dashed] (0,0) circle [radius=2cm] ; 

\coordinate (C1) at (45:1.5cm); \node at (C1) {\large$\Lambda_1$}; 
\coordinate (C3) at (135:1.5cm); \node at (C3) {\large$\Lambda_3$}; 
\coordinate (C5) at (225:1.5cm); \node at (C5) {\large$\Lambda_5$}; 
\coordinate (C7) at (315:1.5cm); \node at (C7) {\large$\Lambda_7$}; 
\end{tikzpicture}
}};

\node at (7.5,0) {
\scalebox{0.8}{
\begin{tabular}{c|l|c|l}
&type &weight $\beta_i$ satisfying $\langle\lambda,\beta_i\rangle<0$ for $\lambda\in\Lambda_i$&multiplicity\\ \hline
$\Lambda_1$&$\Lambda_{\bullet\bullet}$&$(0,-1)$&$n+1$ \\
&&$(-1,0)$&$m+1$ \\ \hline
$\Lambda_2$&$\Lambda_\emptyset$&$(0,-1)$&$n+1$ \\ \hline
$\Lambda_3$&$\Lambda_{\bullet\bullet}$&$(0,-1)$&$n+1$ \\ 
&&$(1,0)$&$m+1$ \\ \hline
$\Lambda_4$&$\Lambda_\emptyset$&$(1,0)$&$m+1$ \\ \hline
$\Lambda_5$&$\Lambda_{\bullet\bullet}$&$(1,0)$&$m+1$ \\
&&$(0,1)$&$n+1$ \\ \hline
$\Lambda_6$&$\Lambda_\emptyset$&$(0,1)$&$n+1$ \\ \hline
$\Lambda_7$&$\Lambda_{\bullet\bullet}$&$(0,1)$&$n+1$ \\ 
&&$(-1,0)$&$m+1$ \\ \hline
$\Lambda_8$&$\Lambda_\emptyset$&$(-1,0)$&$m+1$
\end{tabular}
}};

\end{tikzpicture} 
\end{center}
\end{proof}

\subsection{Rank one MCM modules for the type (V)}

For a Hibi ring which is realized as the Segre product of polynomial rings, the rank one MCM modules have already been studied in \cite[Section~2]{Bru} (see also \cite[Proof of Lemma~3.8]{HN}). 
For the completeness, we give the list of rank one MCM modules over the Hibi ring arising from the poset shown in Figure~\ref{poset_typeV}, which is obtained by restricting the arguments in \cite[Section~2]{Bru} to our situation. 

\begin{proposition}
\label{prop_MCM_typeV}
Let $R$ be the Gorenstein Hibi ring with $\Cl(R)\cong\ZZ^2$ associated with the poset shown in Figure~\ref{poset_typeV}. 
Then, we see that $M_\chi$ is a rank one MCM module if and only if $\chi\in\rmX(G)\cong\ZZ^2$ is contained in the shaded area in Figure~\ref{MCM_typeV}. 
In particular, the grayed area represents conic classes. 
\end{proposition}

\begin{figure}[H]
\begin{center}
\newcommand{\edgewidth}{0.06cm} 
\newcommand{\nodewidth}{0.06cm} 
\newcommand{\noderad}{0.15} 
\newcommand{\pmwidth}{0.65cm} 
\newcommand{\pmcolor}{lightgray} 

{\scalebox{0.35}{
\begin{tikzpicture}

\filldraw [line width=\edgewidth, fill=lightgray] (4,4)--(4,0)--(0,-4)--(-4,-4)--(-4,0)--(0,4)--(4,4) ; 
\draw[pattern=my north west lines, pattern color=blue] (4,8)--(4,4)--(8,4)--(4,0)--(4,-4)--(0,-4)--(-4,-8)--(-4,-4)--(-8,-4)--(-4,0)--(-4,4)--(0,4)--(4,8) ; 
\draw [line width=\edgewidth] (0,4)--(4,8)--(4,4) ; \draw [line width=\edgewidth] (4,4)--(8,4)--(4,0) ;
\draw [line width=\edgewidth] (0,-4)--(-4,-8)--(-4,-4) ; \draw [line width=\edgewidth] (-4,-4)--(-8,-4)--(-4,0) ;

\draw [line width=\edgewidth] (4,0)--(4,-4)--(0,-4) ; \draw [line width=\edgewidth] (-4,0)--(-4,4)--(0,4) ;
\draw[->, line width=0.085cm]  (-10,0)--(10,0); \draw[->, line width=0.085cm]  (0,-8)--(0,8); 

\node at (-1.6,4.5) {\huge $(0,n+1)$} ; \node at (1.8,-4.5) {\huge $(0,-n-1)$} ; 
\node at (5.6,-0.65) {\huge $(n+1,0)$} ; \node at (-5.8,0.65) {\huge $(-n-1,0)$} ; 
\node at (6.4,8) {\huge $(n+1,2n+2)$} ; \node at (10.4,4) {\huge $(2n+2,n+1)$} ; 
\node at (-7,-8) {\huge $(-n-1,-2n-2)$} ; \node at (-11,-4) {\huge $(-2n-2,-n-1)$} ; 

\draw [line width=\nodewidth, fill=black] (0,4) circle [radius=\noderad] ; \draw [line width=\nodewidth, fill=black] (0,-4) circle [radius=\noderad] ; 
\draw [line width=\nodewidth, fill=black] (4,0) circle [radius=\noderad] ; \draw [line width=\nodewidth, fill=black] (-4,0) circle [radius=\noderad] ; 
\draw [line width=\nodewidth, fill=black] (8,4) circle [radius=\noderad] ; \draw [line width=\nodewidth, fill=black] (4,8) circle [radius=\noderad] ; 
\draw [line width=\nodewidth, fill=black] (-8,-4) circle [radius=\noderad] ; \draw [line width=\nodewidth, fill=black] (-4,-8) circle [radius=\noderad] ; 
\end{tikzpicture}
} }
\end{center}
\caption{The region of rank one MCM modules for type (V)}
\label{MCM_typeV}
\end{figure}
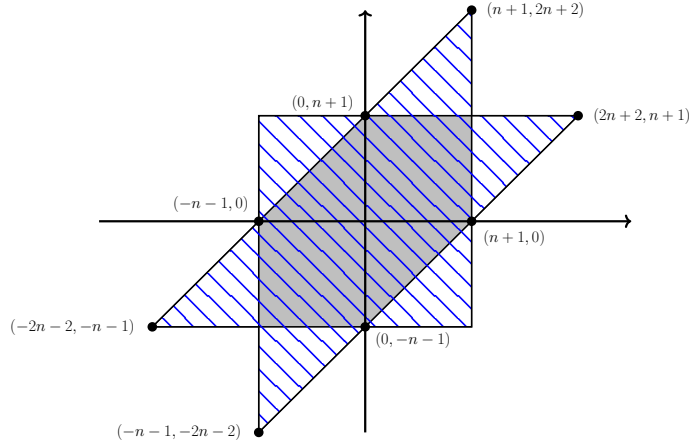

\appendix
\renewcommand{\thesection}{\Alph{section}}
\section{Remarks on splitting NCCRs for toric rings with class group $\ZZ$}
\label{sec_dimer}

In this section, we will discuss splitting NCCRs of a Gorenstein toric ring $R$ with class group $\ZZ$. 
In particular, we focus on the case of $\dim R=3$, in which case any splitting NCCR can be obtained from a consistent dimer model. 

\begin{settings}
\label{set_toric}
$R=\kk[\tau^\vee\cap\ZZ^d]$ is a $d$-dimensional Gorenstein toric ring with $\Cl(R)\cong\ZZ$.
(We remark that using (\ref{cl_seq}) we see that the number of rays of $\tau$ is $n=d+1$.) 
We denote by $T(a)$ a divisorial ideal given by an element $a\in\Cl(R)$. 
Let $\calD_1,\cdots,\calD_{d+1}$ be prime divisors corresponding to rays of $\tau$, and $\beta_i$ denotes the character corresponding to $\calD_i$, and we use the same notation $\beta_i$ for the weight of $\beta_i$. 
We may assume that $\beta_i\in\rmX(G)\cong\ZZ$ satisfies 
\[
\beta_1\le\beta_2\le\cdots\le\beta_s<0<\beta_{s+1}\le\beta_{s+2}\le\cdots\le\beta_{d+1}.
\]
Let $W\coloneqq \bigoplus_{i=1}^{d+1}V_{\beta_i}$. 
Then, $R$ is the ring of invariants $R=S^G$ under the action of $G=\Hom(\Cl(R),\kk^\times)$ on $S=\Sym W\cong\kk[x_1,\cdots,x_{d+1}]$ defined by $g\cdot x_i=\beta_i(g)x_i$ for any $g\in G$. 
Since $R$ is Gorenstein,  the top exterior $\bigwedge^{d+1}W$ is trivial.
Thus, we have that $-(\beta_1+\cdots+\beta_s)=\beta_{s+1}+\cdots+\beta_{d+1}$, and hence the representation $W$ is quasi-symmetric in particular. 

In the following, in order to avoid the triviality, we assume that there are at least two strictly positive and two strictly negative weights $\beta_i$ and the greatest common divisor of weights is one. 
We let $\beta\coloneqq-(\beta_1+\cdots+\beta_s)\in\ZZ_{>0}$. 
\end{settings}

\medskip

We note that a toric ring satisfying Settings~\ref{set_toric} actually admits a splitting NCCR as follows. 
(Since $R$ is quasi-symmetric, this also follows from \cite[Theorem~1.19]{SpVdB}.) 

\begin{theorem}[{see \cite[Theorem~8.9]{VdB2}}]
\label{NCCR_Z}
Let $R$ be a toric ring satisfying Settings~\ref{set_toric}. 
Then, the basic $R$-module $N\coloneqq\bigoplus_{a=0}^{\beta-1}T(a)$ gives a splitting NCCR of $R$. 
In particular, the number of direct summands in $N$ is $\beta$. 
\end{theorem}

In order to study splitting NCCRs further, we note the description of rank one MCM modules. 

\begin{lemma}
\label{MCM_conic}
Let $R$ be a toric ring satisfying Settings~\ref{set_toric}. For $a\in\Cl(R)$, we see that 
$T(a)$ is an MCM $R$-module if and only if $a\in[-\beta+1,\beta-1]\cap\ZZ$. Thus, we see that every rank one MCM module is conic. 
\end{lemma}

\begin{proof}
By \cite[Lemma~8.1]{VdB2} (see also \cite{Sta,VdB1}), we have that $T(a)$ is MCM if and only if $a\in[-\beta+1,\beta-1]\cap\ZZ$. 
On the other hand, $T(a)$ is conic if and only if $T(a)$ is isomorphic to a module of covariants $M_{-\chi}$ for a strongly critical character $\chi$, that is, $\chi=\sum_i\delta_i\beta_i$ in $\rmX(G)_\RR$ with $\delta_i\in(-1,0]$ for all $i$. 
By the assumptions in Setting~\ref{set_toric} such a character satisfies 
$$ -(\beta_{s+1}+\cdots +\beta_{d+1})=-\beta < \chi < \beta=-(\beta_1+\cdots +\beta_s),$$
thus we have the assertion. 
\end{proof}

We then give the precise description of splitting NCCRs as follows. 

\begin{proposition}
\label{NCCR_Z_any}
Let $R$ be a toric ring satisfying Settings~\ref{set_toric}, and $N=\bigoplus_{a=0}^{\beta-1}T(a)$ is the basic module given in Theorem~\ref{NCCR_Z}. 
Then, modules with the form $(N\otimes_RI)^{**}$ where $I$ is a divisorial ideal are precisely basic modules giving splitting NCCRs of $R$. 
\end{proposition}

\begin{proof}
Let $M\coloneqq T(a_1)\oplus T(a_2)\oplus\cdots\oplus T(a_r)$ be a basic module giving a splitting NCCR of $R$, where $a_1, \cdots, a_r\in\ZZ$ for some $r\in\ZZ_{>0}$, and we may assume that $a_1<a_2\cdots<a_r$. 
By the maximality of modules giving an NCCR (see \cite[Proposition~4.5]{IW1}) and the description of rank one MCM modules shown in Lemma~\ref{MCM_conic}, we easily see that $a_{i+1}=a_1+i$ for all $i=0,\cdots,r-1$ and $r=\beta$. 
Thus, we have that $M\cong(N\otimes_RT(a_1))^{**}$. 
On the other hand, since $\End_R((N\otimes_RI)^{**})\cong\End_R(N)$ holds (see e.g., \cite[the proof of Lemma~6.1]{Nak1}), this actually gives a splitting NCCR by Theorem~\ref{NCCR_Z}. 
\end{proof}

We then turn our attention to $3$-dimensional Gorenstein toric rings. 
In this case, splitting NCCRs are obtained from dimer models satisfying the consistency condition (see \cite{Bro,IU}). 
A \emph{dimer model} is a polygonal cell decomposition of the real two-torus whose nodes and edges form a finite bipartite graph. 
Since a dimer model is a bipartite graph, we color each node either black or white, and each edge connects a black node to a white node. 
We say that a dimer model is \emph{reduced} if it does not contain a node whose valency is two. 
If there is a node whose valency is two, we can remove it as shown in \cite[Figure~5]{IU}, and make a dimer model reduced. 
We say that two reduced dimer models are \emph{isomorphic} if their cell decompositions of the real two-torus are homotopy equivalent. 

On the other hand, we can obtain the quiver with potential $(Q_{\Gamma},W_{\Gamma})$ as the dual of a dimer model $\Gamma$. 
Namely, for a dimer model $\Gamma$, we assign a vertex dual to each face in $\Gamma$, an arrow dual to each edge in $\Gamma$. 
Note that the orientation of arrows is given so that the white node is on the right of the arrow. 
The potential $W_{\Gamma}$ is a certain linear combination of cycles surrounding nodes on $\Gamma$. 
For example, Figure~\ref{ex_dimer} is a dimer model and the associated quiver.  

\begin{figure}[H]
\begin{center}
\begin{tikzpicture}
\node (DM) at (0,0) 
{\scalebox{0.6}{
\begin{tikzpicture}
\node (B1) at (1,1){$$}; \node (B2) at (3,1){$$}; \node (W1) at (0,2){$$}; \node (W2) at (2,2){$$};
\node (B3) at (0,3){$$}; \node (B4) at (2,3){$$}; \node (W3) at (1,4){$$}; \node (W4) at (3,4){$$};

\draw[line width=0.035cm]  (-0.5,0.5) rectangle (3.5,4.5);

\draw[line width=0.05cm]  (B1)--(W2)--(B4)--(W3)--(B3)--(W1)--(B1);
\draw[line width=0.05cm]  (W2)--(B2)--(3.5,1.5);\draw[line width=0.05cm] (W2)--(B3);
\draw[line width=0.05cm] (-0.5,1.5)--(W1);\draw[line width=0.05cm] (B4)--(W4)--(3.5,3.5); \draw[line width=0.05cm] (B3)--(-0.5,3.5);
\draw[line width=0.05cm] (B1)--(1,0.5); \draw[line width=0.05cm] (B2)--(3,0.5); 
\draw[line width=0.05cm] (W3)--(1,4.5); \draw[line width=0.05cm] (W4)--(3,4.5); 

\filldraw  [ultra thick, fill=black] (1,1) circle [radius=0.17] ;\filldraw  [ultra thick, fill=black] (3,1) circle [radius=0.17] ;
\filldraw  [ultra thick, fill=black] (0,3) circle [radius=0.17] ;\filldraw  [ultra thick, fill=black] (2,3) circle [radius=0.17] ;
\draw  [ultra thick,fill=white] (0,2) circle [radius=0.18] ;\draw  [ultra thick, fill=white] (2,2) circle [radius=0.18] ;
\draw  [ultra thick, fill=white] (1,4) circle [radius=0.18] ;\draw  [ultra thick, fill=white] (3,4) circle [radius=0.18] ;
\end{tikzpicture}
} }; 

\node (QV) at (5,0) 
{\scalebox{0.6}{
\begin{tikzpicture}[sarrow/.style={black, -latex, very thick}]
\node (B1) at (1,1){$$}; \node (B2) at (3,1){$$}; \node (W1) at (0,2){$$}; \node (W2) at (2,2){$$};
\node (B3) at (0,3){$$}; \node (B4) at (2,3){$$}; \node (W3) at (1,4){$$}; \node (W4) at (3,4){$$};

\node (Q0) at (1,2){$0$}; \node (Q1) at (1,3){$1$}; \node (Q2) at (3.5,2.4){$2$}; \node (Q2a) at (-0.5,2.4){$2$};
\node (Q3a) at (-0.5,0.5){$3$}; \node (Q3b) at (-0.5,4.5){$3$};  \node (Q3c) at (3.5,0.5){$3$};  \node (Q3d) at (3.5,4.5){$3$}; 
\node (Q4a) at (2,0.5){$4$}; \node (Q4b) at (2,4.5){$4$};


\draw[lightgray, line width=0.05cm]  (B1)--(W2)--(B4)--(W3)--(B3)--(W1)--(B1);
\draw[lightgray, line width=0.05cm]  (W2)--(B2)--(3.5,1.5);\draw[lightgray, line width=0.05cm] (W2)--(B3);
\draw[lightgray, line width=0.05cm] (-0.5,1.5)--(W1);\draw[lightgray, line width=0.05cm] (B4)--(W4)--(3.5,3.5); \draw[lightgray, line width=0.05cm] (B3)--(-0.5,3.5);
\draw[lightgray, line width=0.05cm] (B1)--(1,0.5); \draw[lightgray, line width=0.05cm] (B2)--(3,0.5); 
\draw[lightgray, line width=0.05cm] (W3)--(1,4.5); \draw[lightgray, line width=0.05cm] (W4)--(3,4.5); 

\filldraw  [ultra thick, draw=lightgray, fill=lightgray] (1,1) circle [radius=0.17] ;\filldraw  [ultra thick, draw=lightgray, fill=lightgray] (3,1) circle [radius=0.17] ;
\filldraw  [ultra thick,draw=lightgray, fill=lightgray] (0,3) circle [radius=0.17] ;\filldraw  [ultra thick,draw=lightgray, fill=lightgray] (2,3) circle [radius=0.17] ;
\draw  [ultra thick, draw=lightgray,fill=white] (0,2) circle [radius=0.18] ;\draw  [ultra thick, draw=lightgray, fill=white] (2,2) circle [radius=0.18] ;
\draw  [ultra thick, draw=lightgray, fill=white] (1,4) circle [radius=0.18] ;\draw  [ultra thick, draw=lightgray, fill=white] (3,4) circle [radius=0.18] ;

\draw[sarrow, line width=0.06cm] (Q0)--(Q1); \draw[sarrow, line width=0.06cm] (Q1)--(Q2);\draw[sarrow, line width=0.06cm] (Q2)--(Q4a);
\draw[sarrow, line width=0.06cm] (Q4a)--(Q0); \draw[sarrow, line width=0.06cm] (Q0)--(Q3a);\draw[sarrow, line width=0.06cm] (Q1)--(Q3b);
\draw[sarrow, line width=0.06cm] (Q4b)--(Q1); \draw[sarrow, line width=0.06cm] (Q3a)--(Q4a); \draw[sarrow, line width=0.06cm] (Q3b)--(Q4b);
\draw[sarrow, line width=0.06cm] (Q2)--(Q4b);\draw[sarrow, line width=0.06cm] (Q4a)--(Q3c);\draw[sarrow, line width=0.06cm] (Q4b)--(Q3d);
\draw[sarrow, line width=0.06cm] (Q2a)--(Q0);\draw[sarrow, line width=0.06cm] (Q3d)--(Q2);
\draw[sarrow, line width=0.06cm] (Q3b)--(Q2a);\draw[sarrow, line width=0.06cm] (Q3c)--(Q2); \draw[sarrow, line width=0.06cm] (Q3a)--(Q2a);
\end{tikzpicture}
} }; 
\end{tikzpicture}
\caption{An example of a dimer model and the associated quiver}
\label{ex_dimer}
\end{center}
\end{figure}
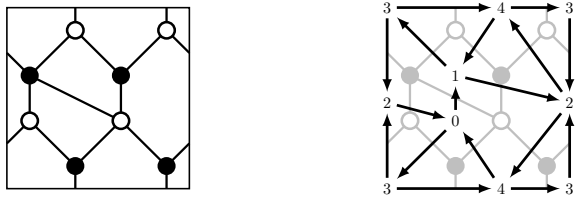

Using the quiver with potential $(Q_{\Gamma},W_{\Gamma})$, we define the Jacobian algebra $\calP(Q_{\Gamma},W_{\Gamma})$, 
which is the path algebra of $Q_{\Gamma}$ with relations coming from $W_{\Gamma}$. 
If $\Gamma$ satisfies the ``consistency condition",  then the center $R\coloneqq\mathrm{Z}\left(\calP(Q_{\Gamma},W_{\Gamma})\right)$ is a $3$-dimensional  Gorenstein toric ring (see \cite[Lemma~5.6]{Bro}). 
Furthermore, there exists a reflexive module $M$ such that $\calP(Q_{\Gamma},W_{\Gamma})\cong\End_R(M)$ and this is a splitting NCCR of $R$ \cite{Bro,IU}. 
We remark that for any $3$-dimensional Gorenstein toric ring, there is a consistent dimer model giving $R$ as the center of the Jacobian algebra 
(see \cite{Gul,IU}), and we call it a consistent dimer model associated with $R$. However, such a consistent dimer model is not unique in general. 

On the other hand, splitting NCCRs always come from consistent dimer models. 
In particular, there is an algorithm to construct a consistent dimer model from a reflexive module giving a splitting NCCR (see \cite{Boc2,CQV}). 
Also, it is known that the number of rank one reflexive modules appearing in basic modules giving splitting NCCRs of $R$ is all the same, and it coincides with the number of faces on consistent dimer models associated with $R$. 
Thus, by combining Proposition~\ref{NCCR_Z_any} we especially have the following. 

\begin{corollary}
\label{cor_splitting_dimer}
Let $R$ be a toric ring satisfying Setting~\ref{set_toric} with $d=3$. 
We denote the number of faces on some (and hence any) consistent dimer model associated with $R$ by $r$. 
Then, any module giving a splitting NCCR of $R$ takes the form $\bigoplus_{i=0}^{r-1}T(a+i)$ with $a\in\ZZ$. 
\end{corollary}

As we mentioned, a consistent dimer model associated with $R$ is not unique. 
However, in our situation the endomorphism ring $\End_R(\bigoplus_{i=0}^{r-1}T(a+i))$ is isomorphic for all $a\in\ZZ$. 
Thus, using the method given in \cite[subsection~5.3]{CQV}, we see that all consistent dimer models reconstructed from modules giving splitting NCCRs are isomorphic. 

\begin{proposition}
\label{unique_dimer}
Let $R$ be a toric ring satisfying Setting~\ref{set_toric} with $d=3$. 
Then, a consistent dimer model associated with $R$ is unique up to isomorphism. 
\end{proposition} 

We then consider the relationships of modules giving splitting NCCRs using the mutations. 
Here, we recall the definition of the mutation. 
Let $M=\bigoplus_{i\in I}M_i$ be a basic reflexive $R$-module giving an NCCR where $I=\{1, \cdots, r\}$. 
For each $i\in I$, we set $M_{I{\setminus}\{i\}}=\bigoplus_{j\in I{\setminus}\{i\}}M_j$. 
We say that a morphism $\varphi:N\rightarrow M_i$ with $N\in\add_R M_{I{\setminus}\{i\}}$ is a \emph{right $(\add_R M_{I{\setminus}\{i\}})$-approximation} of $M_i$ if $\Hom_R(M_{I{\setminus}\{i\}}, \varphi)$ is surjective. 
Also, we say that $\varphi$ is \emph{minimal} if there does not exist a non-zero direct summand of $N$ that is mapped to zero via $\varphi$. 
By \cite[the proof of Proposition~4.2]{AS}, we have that $\add_R M_{I{\setminus}\{i\}}$ is contravariantly finite, thus a minimal right $(\add_R M_{I{\setminus}\{i\}})$-approximation $\varphi$ exists and is unique up to isomorphism. 
We then define the \emph{right mutation} $\mu^+_i$ of $M$ at $i\in I$ as $\mu^+_i(M)\coloneqq M_{I{\setminus}\{i\}}\oplus \Ker\varphi.$ 
Also, we define the \emph{left mutation} $\mu^-_i$ of $M$ at $i\in I$ as $\mu^-_i(M)\coloneqq (\mu^+_i(M^*))^*$, and 
$\mu^+_i(M)=\mu^-_i(M)$ holds when $\dim R=3$ (see \cite[Section~6]{IW1}). Thus, we will simply denote this by $\mu_i(M)$, and call the \emph{mutation} of $M$ at $i\in I$ (or at $M_i$). 
Furthermore, by the results in \cite[Section~6]{IW1}, we have that $\mu_i\mu_i(M)=M$ and $\mu_i(M)$ also gives an NCCR. 
We remark that even if $\End_R(M)$ is a splitting NCCR, $\End_R(\mu_i(M))$ is not splitting in general, this is just an NCCR. 
Repeating the mutations, we obtain many modules giving NCCRs. In general, NCCRs are infinitely many families even if we only consider splitting ones. 
However, the number of generators giving splitting NCCRs is finite up to isomorphism, because if $M$ is such a generator then $M$ is MCM by the definition of NCCR, and the number of rank one MCM modules is finite \cite[Corollary~5.2]{BG1}. 
Under these backgrounds, we introduce the notion of the \emph{exchange graph} of this mutation. 

\begin{definition}
The \emph{exchange graph} $\mathsf{EG}(R)$ (resp. $\mathsf{EG}_0(R)$) of the mutations of modules (resp. generators) giving splitting NCCRs is a graph whose 
vertices are modules (resp. generators) giving splitting NCCRs, and we draw an edge between modules (resp. generators) $M$ and  $M^\prime$ 
giving splitting NCCRs if $M^\prime$ is described as $M^\prime=\mu_i(M)$ for some $i\in I$. 
\end{definition}

In \cite{Nak1}, the author showed that if $R$ is a $3$-dimensional Gorenstein toric ring defined by a reflexive polygon, then $\mathsf{EG}(R)$ is connected. 
However, we do not know whether $\mathsf{EG}(R)$ is connected or not in general (see \cite[Question~6.4]{Nak1}). 
In our situation, we can obtain the next theorem, which gives a partial affirmative answer to this problem. 

\begin{theorem}
\label{ex_graph_dimer}
Let $R$ be a toric ring satisfying Setting~\ref{set_toric} with $d=3$. 
Let $M(a)=\bigoplus_{i=0}^{r-1}T(-a+i)$ be a module giving a splitting NCCR of $R$ where $a\in\ZZ$ (see Corollary~\ref{cor_splitting_dimer}). 
Then, the exchange graph $\mathsf{EG}(R)$ is connected. 
In particular, the exchange graph $\mathsf{EG}_0(R)$ takes the following form. 
\begin{center}
{\scalebox{1}{
\begin{tikzpicture}
\node (Ar) at (8.7,0) {$M(0)$}; \node (Ar1) at (6.7,0) {$M(1)$};
\node (A0) at (-0.3,0) {$M(r-1)$}; \node (A1) at (2,0) {$M(r-2)$};

\draw[thick] (5.2,0)--(Ar1)--(Ar);
\draw[thick] (A0)--(A1)--(3.8,0) ;
\draw[thick,dotted] (4,0)--(5,0);
\end{tikzpicture} }}
\end{center}
\end{theorem}

\begin{proof}
First, we recall that $T(-a)$ is isomorphic to the module of covariants $M_\chi=(S\otimes_\kk V_\chi)^G$ associated with the character $\chi$ of weight $a$. 
(By the abuse of notation, we will write $\chi=a$.) 
Thus, we may write $M(a)=\bigoplus_{i=0}^{r-1}M_{a-i}$. Let $\calL_a=\{a,a-1,\cdots,a-r+1\}$. 
We see that $\chi=a$ is separated from $\calL_a{\setminus}\{a\}$ by $\lambda\in\rmY(G)_\RR$ with $\lambda<0$. 
Since $\dim R=3$ and $\Cl(R)\cong\ZZ$, there are four prime divisors on $\Spec R$. 
Thus, the weights of characters corresponding to prime divisors can be described as 
$\beta_1,\cdots,\beta_4$ satisfying $\beta_1,\beta_2>0$, $\beta_3,\beta_4<0$ and $\beta_1+\beta_2=-(\beta_3+\beta_4)=r$. 
Thus, for a one-parameter subgroup $\lambda<0$, we have that $\langle\lambda,\beta_3\rangle>0$ and $\langle\lambda,\beta_4\rangle>0$. 
Then, by the same arguments shown in Section~\ref{sec_HibiZ2}, we have the exact sequence 
$$
C_{\lambda,a}\colon0\rightarrow V_{a+\beta_3+\beta_4}\otimes_\kk S\xrightarrow{\delta_2}(V_{a+\beta_3}\otimes_\kk S)\oplus (V_{a+\beta_4}\otimes_\kk S)\xrightarrow{\delta_1}V_{a}\otimes_\kk S. 
$$
Then, by Lemma~\ref{key_lem2}, we see that $\Hom_{(G,S)}(P_{\calL_a{\setminus}\{a\}}, \delta_1)$ is surjective. 
Also, we see that $a+\beta_3, a+\beta_4\in\calL_a{\setminus}\{a\}$. 
Thus, applying the functor $(-)^G$ to $C_{\lambda,a}$, we have the exact sequence 
$$
0\rightarrow M_{a+\beta_3+\beta_4}\xrightarrow{\delta_2^G}M_{a+\beta_3}\oplus M_{a+\beta_4}\xrightarrow{\delta_1^G}M_{a}, 
$$
and $\delta_1^G$ is a right ($\add_R\bigoplus_{i=1}^{r-1}M_{a-i}$)-approximation of $M_a$. 
Since $M_{a+\beta_3+\beta_4}=M_{a-r}$, we see that $M(a-1)=\bigoplus_{i=1}^{r}M_{a-i}$ is the mutation of $M(a)$ at $M_a$, 
and hence we can draw an edge between $M(a)$ and $M(a-1)$ for any $a\in\ZZ$. 
(Similarly, using $\lambda\in\rmY(G)_\RR$ with $\lambda>0$, we can see that $M(a)$ is the mutation of $M(a-1)$ at $M_{a-r}$, in which case we consider the sequence $C_{\lambda,a-r}$ using $\beta_1$ and $\beta_2$.) 

Since $M(a)$ is a generator if and only if $a=0,1,\cdots,r-1$, we have the exchange graph $\mathsf{EG}_0(R)$ as above. 
\end{proof}

\begin{example}

We consider the $3$-dimensional Gorenstein toric ring $R$ defined by the cone $\sigma$: 
\[
\sigma=\mathrm{Cone}\{\bfv_1=(1,-1,1), \bfv_2=(0,1,1), \bfv_3=(-1,0,1), \bfv_4=(-1,-1,1) \}. 
\]
As an element in $\Cl(R)$, we have that $\calD_2+2\calD_1=0$, $\calD_3-4\calD_1=0$, $\calD_4+3\calD_1=0$ where $\calD_i$ is the prime divisor corresponding to $\bfv_i$. 
Therefore, we have that $\Cl(R)=\langle\calD_1\rangle\cong\ZZ$, and each divisorial ideal is represented by $T(a)$ where $a\in\ZZ$. 
For this toric ring $R$, there is a unique consistent dimer model associated with $R$ (see Proposition~\ref{unique_dimer}), and that is the one given in Figure~\ref{ex_dimer} (see \cite[Subsection~5.5]{Nak1}). 
By Corollary~\ref{cor_splitting_dimer}, generators giving splitting NCCRs are $M(a)=T(-a)\oplus T(-a+1)\oplus\cdots\oplus T(-a+4)$ with $a=0,\cdots,4$. 
By Theorem~\ref{ex_graph_dimer}, we have the exchange graph $\mathsf{EG}_0(R)$ as follows. 
In the following figure, the double circle stands for the origin, and each point $a\in\ZZ$ corresponds to the module $T(a)$. 

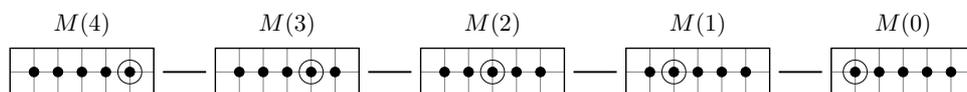
\begin{figure}[H]
\begin{center}
{\scalebox{0.9}{
\begin{tikzpicture}

\node  at (0,0.7) {$M(4)$}; 
\node (A1) at (0,0)
{\scalebox{0.35}{\begin{tikzpicture}
\draw [step=1,thin, gray] (0,-1) grid (6,1); 
\draw [ultra thick] (0,-1) rectangle (6,1);
\draw (5,0) circle [ultra thick, radius=0.5]; 
\filldraw [thick, fill=black] (5,0) circle [radius=0.2]; 
\filldraw  [thick, fill=black] (4,0) circle [radius=0.2] ;
\filldraw  [thick, fill=black] (3,0) circle [radius=0.2] ;
\filldraw  [thick, fill=black] (2,0) circle [radius=0.2] ;
\filldraw  [thick, fill=black] (1,0) circle [radius=0.2] ;
\end{tikzpicture}}};

\node at (3,0.7) {$M(3)$}; 
\node (A4) at (3,0)
{\scalebox{0.35}{\begin{tikzpicture}
\draw [step=1,thin, gray] (0,-1) grid (6,1); 
\draw [ultra thick] (0,-1) rectangle (6,1);
\draw (4,0) circle [ultra thick, radius=0.5]; 
\filldraw [thick, fill=black] (5,0) circle [radius=0.2]; 
\filldraw  [thick, fill=black] (4,0) circle [radius=0.2] ;
\filldraw  [thick, fill=black] (3,0) circle [radius=0.2] ;
\filldraw  [thick, fill=black] (2,0) circle [radius=0.2] ;
\filldraw  [thick, fill=black] (1,0) circle [radius=0.2] ;
\end{tikzpicture}}};

\node at (6,0.7) {$M(2)$}; 
\node (A2) at (6,0)
{\scalebox{0.35}{\begin{tikzpicture}
\draw [step=1,thin, gray] (0,-1) grid (6,1); 
\draw [ultra thick] (0,-1) rectangle (6,1);
\draw (3,0) circle [ultra thick, radius=0.5]; 
\filldraw [thick, fill=black] (5,0) circle [radius=0.2]; 
\filldraw  [thick, fill=black] (4,0) circle [radius=0.2] ;
\filldraw  [thick, fill=black] (3,0) circle [radius=0.2] ;
\filldraw  [thick, fill=black] (2,0) circle [radius=0.2] ;
\filldraw  [thick, fill=black] (1,0) circle [radius=0.2] ;
\end{tikzpicture}}};

\node at (9,0.7) {$M(1)$}; 
\node (A3) at (9,0)
{\scalebox{0.35}{\begin{tikzpicture}
\draw [step=1,thin, gray] (0,-1) grid (6,1); 
\draw [ultra thick] (0,-1) rectangle (6,1);
\draw (2,0) circle [ultra thick, radius=0.5]; 
\filldraw [thick, fill=black] (5,0) circle [radius=0.2]; 
\filldraw  [thick, fill=black] (4,0) circle [radius=0.2] ;
\filldraw  [thick, fill=black] (3,0) circle [radius=0.2] ;
\filldraw  [thick, fill=black] (2,0) circle [radius=0.2] ;
\filldraw  [thick, fill=black] (1,0) circle [radius=0.2] ;
\end{tikzpicture}}};

\node at (12,0.7) {$M(0)$}; 
\node (A0) at (12,0)
{\scalebox{0.35}{\begin{tikzpicture}
\draw [step=1,thin, gray] (0,-1) grid (6,1); 
\draw [ultra thick] (0,-1) rectangle (6,1);
\draw (1,0) circle [ultra thick, radius=0.5]; 
\filldraw [thick, fill=black] (5,0) circle [radius=0.2]; 
\filldraw  [thick, fill=black] (4,0) circle [radius=0.2] ;
\filldraw  [thick, fill=black] (3,0) circle [radius=0.2] ;
\filldraw  [thick, fill=black] (2,0) circle [radius=0.2] ;
\filldraw  [thick, fill=black] (1,0) circle [radius=0.2] ;
\end{tikzpicture}}};

\draw[thick] (A1)--(A4)--(A2)--(A3)--(A0) ;

\end{tikzpicture} }}
\caption{An example of the exchange graph $\mathsf{EG}_0(R)$}
\label{ex_exchange}
\end{center}
\end{figure}
\end{example}

\subsection*{Acknowledgement}
The author would like to thank Akihiro Higashitani for valuable discussions and comments, especially the proof of Lemma~\ref{poset_Z2} is based on his idea. 
The author also thanks \v{S}pela \v{S}penko for stimulating discussions concerning NCCRs. 
The author thank the anonymous referee for valuable comments and suggestions. 

The author is supported by World Premier International Research Center Initiative (WPI initiative), MEXT, Japan, and JSPS Grant-in-Aid for Young Scientists (B) 17K14159.



\begin{thebibliography}{999}
\bibitem[AS]{AS} M. Auslander and S. O. Smal{\o}, \emph{Preprojective modules over Artin algebras}, J. Algebra \textbf{66} (1980), 61--122. 
\bibitem[Boc]{Boc2} R. Bocklandt, \emph{Generating toric noncommutative crepant resolutions}, J. Algebra \textbf{364} (2012), 119--147. 
\bibitem[Bol]{Bol} B. Bollob\'as, \emph{Modern graph theory}, Graduate Texts in Mathematics, \textbf{184} Springer-Verlag, New York, 1998. 
\bibitem[Bro]{Bro} N. Broomhead, \emph{Dimer model and Calabi-Yau algebras}, Mem. Amer. Math. Soc. \textbf{215}  no. 1011, (2012). 
\bibitem[Bru]{Bru} W. Bruns, \emph{Conic divisor classes over a normal monoid algebra}, Commutative algebra and algebraic geometry, Contemp. Math. \textbf{390}, Amer. Math. Soc., (2005), 63-71. 
\bibitem[BG1]{BG1} W. Bruns and J. Gubeladze, \emph{Divisorial linear algebra of normal semigroup rings}, Algebr. Represent. Theory \textbf{6} (2003), no. 2, 139--168. 
\bibitem[BG2]{BG2} W. Bruns and J. Gubeladze, \emph{Polytopes, rings and K-theory}, Springer Monographs in Mathematics. Springer, Dordrecht, (2009).  
\bibitem[BLVdB]{BLVdB} R.-O. Buchweitz, G. J. Leuschke, and M. Van den Bergh, \emph{Non-commutative desingularization of determinantal varieties II: arbitrary minors}, Int. Math. Res. Not. IMRN \textbf{9}, 2748--2812 (2016). 
\bibitem[CQV]{CQV} A. Craw and A. Q. V\'elez, \emph{Cellular resolutions of noncommutative toric algebras from superpotentials}, Adv. Math. \textbf{229} (2012), 1516--1554. 
\bibitem[DITW]{DITW} H. Dao, O. Iyama, R. Takahashi, and M. Wemyss, \emph{Gorenstein modifications and $\QQ$-Gorenstein rings}, arXiv:1611.04137. 
\bibitem[FMS]{FMS} E. Faber, G. Muller, and K. E. Smith, \emph{Non-Commutative Resolutions of Toric Varieties}, arXiv:1805.00492. 
\bibitem[Gul]{Gul} D. R. Gulotta, \emph{Properly ordered dimers, $R$-charges, and an efficient inverse algorithm}, J. High Energy Phys. (2008), no. 10, 014, 31. 
\bibitem[Har]{Har} W. Hara, \emph{Non-commutative crepant resolution of minimal nilpotent orbit closures of type A and Mukai flops}, Adv. Math. \textbf{318} (2017), 355--410. 
\bibitem[HHN]{HHN} M. Hashimoto, T. Hibi, and A. Noma, \emph{Divisor class groups of affine semigroup rings associated with distributive lattices}, J. Algebra \textbf{149} (1992), no. 2, 352--357. 
\bibitem[Hib]{Hibi} T. Hibi, \emph{Distributive lattices, affine semigroup rings and algebras with straightening laws}, In:``Commutative Algebra and Combinatorics'' (M. Nagata and H. Matsumura, Eds.), Adv. Stud. {\em Pure Math.} {\bf 11}, North--Holland, Amsterdam, (1987), 93--109. 
\bibitem[HN]{HN} A. Higashitani and Y. Nakajima, \emph{Conic divisorial ideals of Hibi rings and their applications to non-commutative crepant resolutions}, arXiv:1702.07058. 
\bibitem[IU]{IU} A. Ishii and K. Ueda, \emph{Dimer models and the special McKay correspondence}, Geom. Topol. \textbf{19} (2015) 3405--3466. 
\bibitem[Iya]{Iya} O. Iyama, \emph{Auslander correspondence}, Adv. Math. \textbf{210} (2007), no. 1, 51--82.  
\bibitem[IR]{IR} O. Iyama and I. Reiten, \emph{Fomin-Zelevinsky mutation and tilting modules over Calabi-Yau algebras}, Amer. J. Math. \textbf{130} (2008), no. 4, 1087--1149. 
\bibitem[IW1]{IW1} O. Iyama and M. Wemyss, \emph{Maximal modifications and Auslander-Reiten duality for non-isolated singularities}, Invent. Math. \textbf{197} (2014), no. 3, 521--586.
\bibitem[IW2]{IW2} O. Iyama and M. Wemyss, \emph{Reduction of triangulated categories and maximal modification algebras for $cA_n$ singularities}, J. Reine Angew. Math. \textbf{738} (2018), 149--202. 
\bibitem[Leu]{Leu} G. J. Leuschke, \emph{Non-commutative crepant resolutions: scenes from categorical geometry}, Progress in commutative algebra 1, 293--361, de Gruyter, Berlin (2012). 
\bibitem[Nak]{Nak1} Y. Nakajima, \emph{Mutations of splitting maximal modifying modules: The case of reflexive polygons}, to appear in Int. Math. Res. Not. IMRN, arXiv:1601.05203. 
\bibitem[SmVdB]{SmVdB} K. E. Smith and M. Van den Bergh, \emph{Simplicity of rings of differential operators in prime characteristic}, Proc. London Math. Soc. (3) \textbf{75} (1997), no. 1, 32--62.
\bibitem[\v{S}pVdB1]{SpVdB} \v{S}. \v{S}penko and M. Van den Bergh, \emph{Non-commutative resolutions of quotient singularities for reductive groups}, Invent. Math. \textbf{210} (2017), no. 1, 3--67.  
\bibitem[\v{S}pVdB2]{SpVdB_toric1} \v{S}. \v{S}penko and M. Van den Bergh, \emph{Non-commutative crepant resolutions for some toric singularities I}, arXiv:1701.05255.
\bibitem[\v{S}pVdB3]{SpVdB_toric2} \v{S}. \v{S}penko and M. Van den Bergh,  \emph{Non-commutative crepant resolutions for some toric singularities II}, arXiv:1707.08245.
\bibitem[Sta]{Sta} R. P. Stanley, \emph{Linear Diophantine equations and local cohomology}, Invent. Math. \textbf{68} (1982), no. 2, 175--193. 
\bibitem[VdB1]{VdB1} M. Van den Bergh, \emph{Cohen-Macaulayness of semi-invariants for tori}, Trans. Amer. Math. Soc. \textbf{336} (1993), no. 2, 557--580. 
\bibitem[VdB2]{VdB2} M. Van den Bergh, \emph{Non-Commutative Crepant Resolutions}, The Legacy of Niels Henrik Abel, Springer-Verlag, Berlin, (2004), 749--770.
\end{thebibliography}
\end{document}